\documentclass[12pt]{amsart} 
\usepackage{amssymb,amsfonts}
\usepackage[all,arc]{xy}
\usepackage{enumerate}
\usepackage{mathrsfs}
\usepackage[linktoc=all]{hyperref}
\hypersetup{
	colorlinks,
	citecolor=black,
	filecolor=black,
	linkcolor=black,
	urlcolor=black
}

\newtheorem{thm}{Theorem}[section]
\newtheorem{cor}[thm]{Corollary}
\newtheorem{prop}[thm]{Proposition}
\newtheorem{lem}[thm]{Lemma}

\newtheorem{fact}[thm]{Fact}

\theoremstyle{definition}
\newtheorem{defn}[thm]{Definition}

\newtheorem{que}[thm]{Question}

\theoremstyle{remark}
\newtheorem{rem}[thm]{Remark}

\makeatletter
\let\c@equation\c@thm
\makeatother
\numberwithin{equation}{section}

\usepackage{cite}
\title{Deformations of Totally Geodesic Foliations and Minimal Surfaces in Negatively Curved 3-Manifolds}

\author{Ben Lowe}
\email{benl@princeton.edu} 
\begin{document}
	
	\begin{abstract}
		
		Let  $g_t$ be a smooth 1-parameter family of negatively curved metrics on a closed hyperbolic 3-manifold $M$ starting at the hyperbolic metric.  We construct foliations of the Grassmann  bundle $Gr_2(M)$ of tangent 2-planes  whose leaves are (lifts of) minimal surfaces in $(M,g_t)$.  These foliations are deformations of the foliation of $Gr_2(M)$ by (lifts of) totally geodesic planes projected down from the universal cover $\mathbb{H}^3$. Our construction continues to work as long as the sum of the squares of the principal curvatures of the  (projections to $M$) of the leaves remains pointwise smaller in magnitude than the ambient Ricci curvature in the normal direction.  In the second part of the paper we give some applications and construct negatively curved metrics for which $Gr_2(M)$ cannot admit a foliation as above.   \iffalse and construct negatively curved metrics for which $Gr_2(M)$ cannot admit a foliation as above.\fi

		%	We construct foliations of the Grassmann bundle $Gr_2(M)$ of tangent 2-planes to a closed hyperbolic 3-manifold whose leaves are (lifts of) minimal surfaces in a negatively curved metric on $M$.  These foliations are deformations of the foliation of $Gr_2(M)$ by (lifts of) totally geodesic planes projected down from the universal cover $\mathbb{H}^3$.  	The set of negatively curved metrics for which we can construct such foliations is analogous to the almost-Fuchsian locus in the space of quasi-Fuchsian metrics on $\Sigma_g \times \mathbb{R}$ introduced in \cite{uhl}. As an application, we prove some density and uniqueness results for stable minimal surfaces in $M$.   

		%	The set $\Omega$ of negatively curved metrics for which we can construct such foliations is analogous to the almost-Fuchsian locus in the space of quasi-Fuchsian metrics on $\Sigma_g \times \mathbb{R}$.
		
		%Finally, we give an example of a negative curved metric on a hyperbolic 3-manifold $M$ for which $Gr_2(M)$ cannot admit a foliation as above.  

		% For metrics to which our construction applies, we show that tangent planes of the $\Sigma$ are dense in $Gr_2(M)$.  In the case that $M$ contains no proper totally geodesic surfaces, we prove a quantitative version of the density. 

		% , and we conjecture that density of the $\Sigma$ in $Gr_2(M)$ fails for this metric.  
		
		\end{abstract} 
		
\maketitle

\section{Introduction}

\subsection{Introduction}
The following statement is a special case of the geodesic rigidity theorem proved by Gromov \cite{georigidity}.    

\begin{thm} \label{geodesicrigidity} 
	Let $M$ be a closed hyperbolic manifold, and let $N$ be a negatively curved Riemannian manifold homeomorphic to $M$. Let $G(M)$ and $G(N)$ be the foliations of the unit tangent bundles $UT(M)$ and $UT(N)$ of $M$ and $N$ by the orbits of the geodesic flow.  Then there is a homeomorphism between $UT(M)$ and $UT(N)$ sending leaves of $G(M)$ to leaves of $G(N)$.    
\end{thm}

In this paper, we study the extent to which a version of this theorem holds when geodesics, which are one-dimensional minimal surfaces, are replaced by two-dimensional minimal surfaces. We restrict ourselves to three ambient dimensions because minimal surfaces in that dimension are better behaved and understood. %, although some of what we do works in any dimension. 

Theorem \ref{geodesicrigidity} implies that many properties of the geodesic flow for an arbitrary  negatively curved metric on $M$ are controlled by the constant curvature geodesic flow.  Much of our interest in trying to prove a minimal surface analogue is that it will allow us to use homogeneous dynamics to study how minimal surfaces in variable negative curvature are distributed in the ambient space. The idea to use homogeneous dynamics in this setting is recent and due to Calegari-Marques-Neves \cite{cmn}.  

%We also hope that Theorem \ref{mainintro} might be useful for studying the space of negatively curved metrics on $M$.     

%This perspective is essentially due to \cite{cmn}.   

 If $M$ is a closed hyperbolic 3-manifold, then the Grassmann bundle of tangent 2-planes to $M$ has a natural foliation by immersed totally geodesic planes.  Denote this foliation by $\mathcal{F}$, and let $g_{hyp}$ be a hyperbolic (constant curvature $-1$) metric on $M$ (Mostow rigidity says that there is a unique such metric up to isometry.)  The following theorem is the main result of this paper.    

\begin{thm} \label{mainintro} 
	Let $\{g_t: t \in [0,1]\}$ be a smooth 1-parameter family of negatively curved metrics on $M$ with $g_0=g_{hyp}$. Then there exists $T \in (0,1] \cup \{\infty\}$ such that for all $t<T$ contained in the interval $[0,1]$ there is a foliation $\mathcal{F}_t$ of $Gr_2(M)$ whose leaves are immersed minimal planes in $(M,g_t)$ lifted to $Gr_2(M)$ by their tangent planes.  Moreover, there is a self-homeomorphism $\Phi$ of $Gr_2(M)$ that sends leaves of $\mathcal{F}$ to leaves of  $\mathcal{F}_t$.

	If $T<\infty$, then for every sequence $t_n \nearrow T$ there exist immersed minimal planes $S_n$ in $(M,g_{t_n})$ which lift to leaves of $\mathcal{F}_{t_n}$ such that the following quantity tends to zero from below for a sequence of points $p_n \in S_n$:        
	\begin{equation} \label{convexintro} 
	  |A_n|^2 + Ric_n ( \nu_n,\nu_n) .    
	\end{equation} 
	Here $\nu_n$ is the unit normal vector to $S_n$ at $p_n$, $A_n$ is the second fundamental form of $S_n$ at $p_n$,  and $Ric_n$ is the Ricci curvature tensor of $g_{t_n}$ at $p_n$.  
	\end{thm}   
	
	\begin{rem} 
		The same theorem should actually hold for all complete hyperbolic 3-manifolds $M$, with the appropriate bounded geometry condition on the family $g_t$.  The assumption that the action of $\pi_1(M)$ on $\mathbb{H}^3$ has a compact fundamental domain makes some of the proofs simpler but does not seem to be essential.  Because of the applications we have in mind, though, we restrict ourselves to the closed case.     	
		
		\end{rem} 
		
		\begin{rem} 
			
		\iffalse 	A previous version of this paper claimed to construct negatively curved metrics on certain closed hyperbolic 3-manifolds $M$ for which there could not exist a foliation as in Theorem \ref{mainintro}. That construction contained a mistake.  For all the author knows right now all negatively curved metrics on $M$ could admit a foliation as in the theorem, although that seems unlikely.\fi  %We expect that the analogue of the geodesic rigidity of Theorem \ref{geodesicrigidity} only holds in our setting for negatively curved metrics in some neighborhood of the hyperbolic metric, but we are unable to prove this.  

			In Section \ref{counterexample}, we construct negatively curved metrics on certain closed hyperbolic 3-manifolds $M$ for which there cannot exist a foliation as in Theorem \ref{mainintro}. The analogue of the geodesic rigidity of Theorem \ref{geodesicrigidity} therefore in general only holds in our setting in some neighborhood of the hyperbolic metric.  % Lower bounds (which we expect to be far from optimal) for the size of this neighborhood follow from Theorem \ref{stability} of the final section.  
			
			\end{rem}

In words, our construction of the foliations $\mathcal{F}_t$ continues to work as long as the sum of the squares of the principal curvatures (of the projections to $M$) of the leaves of our foliations remains pointwise less than the absolute value of the ambient Ricci curvature in the normal direction.   The proof of Theorem \ref{mainintro} occupies Section \ref{mainsection}, where we prove a more intrinsic formulation of it (Theorem \ref{main}), and then modify that proof to give a proof of Theorem \ref{mainintro}.  The proof is loosely speaking a method of continuity argument, where we work in the universal cover and follow the approach of Anderson \cite{and2} to construct properly embedded minimal planes.  Surfaces for which the quantity (\ref{convexintro}) is negative have small mean-convex neighborhoods, which we use to rule out the existence of minimal planes other than the ones from our construction.  These would lead to gaps in the foliations we are trying to construct.   

In \cite{groapp}, Gromov proved a stability result for the totally geodesic foliation $\mathcal{F}$ that applies to metrics $g$ with sectional curvatures pinched close to $-1$.  For these metrics, he constructs an immersed almost-totally-geodesic $g$-minimal plane in $M$ for each leaf of $\mathcal{F}$.  His construction also follows \cite{and2} but, in contrast to this paper, works for closed hyperbolic manifolds of all dimensions and is based on Allard's regularity theorem.  This paper grew out of attempts to find a more direct proof of Gromov's results in dimension 3.

\subsection{Almost-Fuchsian Manifolds} 	
Theorem \ref{mainintro} was motivated by the theory of almost-Fuchsian manifolds.   A homeomorphism $f:S^2 \rightarrow S^2$ is \textit{K-quasiconformal} if for any ball $B(x,r)\subset S^2$ there exists $r'>0$ such that $B(f(x),r')\subset f(B(x,r)) \subset B(f(x),Kr')$.  Here $B(x,s)$ denotes the ball centered at $x$ of radius $s$ in the round metric on $S^2$. The homeomorphism $f$ is \textit{quasiconformal} if it is $K$-quasiconformal for some $K$.   For $\Sigma$ a surface of genus greater than one, a hyperbolic metric on $\Sigma \times \mathbb{R}$ is given by a properly discontinuous action of $\pi_1(\Sigma)$ on $\mathbb{H}^3$ by isometries.  The \textit{limit set} of a hyperbolic metric on $\Sigma \times \mathbb{R}$ is the set of accumulation points of any orbit under this action in the boundary at infinity $ \partial_{\infty}\mathbb{H}^3\cong S^2$.  We say that a hyperbolic metric on $\Sigma \times \mathbb{R}$ is \textit{quasi-Fuchsian} if its limit set is the image of the equator under a quasiconformal homeomorphism of $S^2$. \iffalse(A $C^1$ map $S^2 \rightarrow S^2$ is \textit{K-quasiconformal} if its differential maps circles to ellipses whose ratio between major and minor axes is bounded above by $K$, and .)\fi  The space of all quasi-Fuchsian metrics is parametrized by a product of Teichmuller spaces which correspond to the conformal structures on the two ends.  
	
In \cite{uhl}, Uhlenbeck proved a rigidity theorem for quasi-Fuchsian manifolds which admit an embedded minimal surface with principal curvatures less than 1 in magnitude: that such a quasi-Fuchsian manifold is uniquely determined by the conformal class of the induced metric on this minimal surface and a quadratic differential equivalent to its second fundamental form.  These manifolds are called \textit{almost-Fuchsian}, and have been well studied since (\cite{almostfuchsian}, \cite{zenononuniq},  \cite{andrewopen}, \cite{seppi}, \cite{taubesaf} .)  

The set of metrics to which Theorem \ref{mainintro} applies inside the space of all negatively curved metrics on $M$ is analogous to the set of almost-Fuchsian metrics on $\Sigma \times \mathbb{R}$ inside the space of quasi-Fuchsian metrics, insofar as the existence of minimal surfaces with curvatures bounded by ambient curvatures allows for much greater control.  This makes it possible, for instance, to prove uniqueness statements for the minimal surfaces in question. Guided by this analogy, in Section \ref{counterexample} we construct negatively curved metrics for which foliations as in the statement of Theorem \ref{mainintro} cannot possibly exist. %These examples are based on the existence of quasi-Fuchsian manifolds with multiple $pi_1$-injective embedded minimal surfaces.  
%The version of the geodesic rigidity of Theorem \ref{geodesicrigidity} therefore only holds in our setting in some neighborhood of the constant curvature metric.  Lower bounds (which we expect to be far from optimal) for the size of this neighborhood follow from Theorem \ref{stability} of the final section.  

   % Since the metric we construct can be joined to the constant curvature metric through a path of negatively curved metrics, this shows that the case $T<1$ of Theorem \ref{mainintro} actually occurs.

\subsection{Applications}

We now describe some applications. Principal among these is the following density result for stable properly immersed minimal surfaces in a closed Riemannian 3-manifold $M$ which admits a foliation as in Theorem \ref{mainintro}.  Kahn and Markovic showed  that for every closed hyperbolic 3-manifold $M$, subgroups of $\pi_1(M)$ isomorphic to the fundamental group of a closed surface, or \textit{surface subgroups},  exist in great profusion (\cite{km},\cite{kmcounting}.)  (See also \cite{ursulah}, which gives a more geometric version of Kahn-Markovic's construction and generalizes their results to cocompact lattices in all rank one symmetric spaces except for hyperbolic spaces of even dimension.)   Fixing a metric $g$ on $M$, each of these surface subgroups gives rise by \cite{sacksuhl} or \cite{schoenyauincompressible} to a stable properly immersed minimal surface whose fundamental group includes as a subgroup of $\pi_1(M)$ conjugate to that surface subgroup.  

Let $C$ be a circle in $\partial_\infty{\mathbb{H}^3} \cong S^2$ such that the geodesic plane $P$ in $\mathbb{H}^3$ with limit set $C$ has dense projection to the closed hyperbolic 3-manifold $M$ under the universal covering map. Ratner and Shah independently proved that every geodesic plane $P$ either projects to a dense subset of $M$ whose tangent planes are dense in $Gr_2(M)$ or a closed properly immersed surface (\cite{ratner},\cite{shah}.) (See also \cite{mmo} for a nice proof of this fact.) Let $\Gamma_n$ be a sequence of surface subgroups of $\pi_1(M)$ with limit sets $K_n$-quasicircles Hausdorff converging to  $C$ with $K_n$ tending to 1. (A K-quasicircle is the image of a round circle under a $K$-quasiconformal self-homeomorphism of $S^2$.)  The existence of such a sequence of $\Gamma_n$ for each $C$ follows from \cite{km}.  Let $\Sigma_n$ be a sequence of stable immersed minimal surfaces in $(M,g)$ whose fundamental groups include to the conjugacy classes of the $\Gamma_n$.

%. By $\textit{Kahn-Markovic minimal surface}$ we mean any stable properly immersed minimal surface in $(M,g)$ whose fundamental group injectively includes as one of these groups.  Let $\Sigma_n$ be a sequence of Kahn-Markovic minimal surfaces that have lifts to $\mathbb{H}^3$ with limit sets in $\partial_{\infty} \mathbb{H}^3$ tending in Haudorff distance to a round circle.   

\begin{thm} \label{densityintro} 
Suppose that $g$ can be joined to $g_{hyp}$ by a smooth family of negatively curved metrics parametrized by $[0,1]$ to which Theorem \ref{mainintro} applies with $T=\infty$. Let $\Sigma_n$ be a sequence as above. Then for every open set $U$ in $Gr_2(M)$ there exists a number $N$ so that $\Sigma_n$ has a tangent plane in $U$ for every $n>N$.       
\end{thm} 

\begin{rem} 
This is a slightly stronger statement than simply that the tangent planes of all closed stable immersed minimal surfaces are dense in $Gr_2(M)$, which could be obtained without using the Ratner-Shah theorem mentioned above.  
\end{rem} 
 
A natural question is whether a similar density result is true for all negatively curved metrics on $M$. For negatively curved metrics on $M$ which cannot admit foliations as in Theorem \ref{mainintro}, like those constructed in Section \ref{counterexample} we believe it is possible that a sequence of $\Sigma_n$ as above might fail to be dense in $Gr_2(M)$.    

If $g=g_{hyp}$, it follows from \cite{seppi} that if the limit set of $\Gamma_n$ is a $K$-quasicircle for $K$ sufficiently close to 1, then $\Sigma_n$ is the unique  minimal surface whose fundamental group injectively includes as a subgroup conjugate to $\Gamma_n$. In Section \ref{applications}, we prove a similar uniqueness result for $g$ to which Theorem \ref{mainintro} applies to produce a foliation:  
\begin{thm} \label{uniqintro} 
Suppose that $g$ satisfies the hypotheses of the previous theorem. Then there exists $\delta>0$ such that the following is true.  Suppose the limit set of a surface subgroup $\Gamma$ of $\pi_1(M)$ in $\partial_\infty \mathbb{H}^3$ is a $K$-quasicircle for $K<1+\delta$.  Then there is a unique $g$-minimal surface in $M$ whose fundamental group injectively includes in $\pi_1(M)$ as a subgroup conjugate to $\Gamma$.      
\end{thm} 

\subsection{Quantitative Density} 
%Suppose that $M$ has no properly immersed totally geodesic surfaces in its hyperbolic metric.   Examples of closed hyperbolic 3-manifolds without properly immersed totally geodesic surfaces are described in  ~\cite[Chapter 5]{ah3mtext}. 

Let $\Sigma_n'$ be a sequence of minimal surfaces in $(M,g)$ with areas tending to infinity such that the limit sets of the $\pi_1(\Sigma_n)$ in $\partial_\infty \mathbb{H}^3$ are $K_n$-quasicircles with $K_n$ tending to 1.  Let $\Sigma_n$ be the corresponding sequence of minimal surfaces in the hyperbolic metric on $M$.  

Let $\mu_n'$ be the probability measures on $Gr_2(M)$ that correspond to averaging over the lift of $\Sigma_n'$ to $Gr_2(M)$ in the area form for the metric on $\Sigma_n'$ induced by $(M,g)$.  Let $\mu_n$ be the corresponding measures for the $\Sigma_n$. \iffalse We say that a sequence of probability measures has the \textit{full support property} if every weak-$*$ limit has full support (assigns positive measure to every non-empty open set.) \fi We are able to prove the following quantitative version of Theorem \ref{densityintro}.    

\begin{thm} \label{quantitativeintro} 
Let $g$ satisfy the assumptions of Theorem \ref{densityintro}. Suppose that the $\mu_n$ weak-$*$ converge to the uniform measure on $Gr_2(M,g_{hyp})$.  Then the $\mu_n'$ weak-$*$ converge to a measure $\mu_g$ on $Gr_2(M,g)$ with full support. 
\end{thm} 

If $\mathcal{F}_g$ is the foliation from Theorem \ref{mainintro}, then the measure $\mu_g$ is the sum of a transverse invariant measure for $\mathcal{F}_g$ and the Riemannian area forms for the metrics on the leaves of $\mathcal{F}_g$ induced by their projections to $(M,g)$--- see Section \ref{stronger} for more details. %Our proof uses Garnett's theory of harmonic measures \cite{garn}.  

\begin{rem} \label{mozshah} 
Mozes-Shah  \cite{mozesshah} proved that any sequence of totally geodesic $\Sigma_n$ in $(M,g_{hyp})$ with area tending to infinity becomes uniformly distributed in $Gr_2(M)$.  Theorem \ref{quantitativeintro} therefore applies when the $\Sigma_n$ are all totally geodesic.  In this case the $\Sigma_n'$ are (projections of) leaves of our foliation.

 %The arguments used to prove Theorem \ref{quantitativeintro} %, together with the map $\tilde{f}_\Sigma$ constructed in the proof of Theorem \ref{uniqintro}, 
%can also be used to show that weak-$*$ limits of minimal surfaces corresponding to the $\Sigma_n$ have full support in $Gr_2(M)$ in metrics for which Theorem \ref{mainintro} constructs a foliation.  
	
	\end{rem}

In Theorem \ref{constcurvature} below we prove that the $\mu_n$ always converge to the uniform measure provided that $M$ contains no properly immersed totally geodesic surfaces in its hyperbolic metric.  The condition that $M$ have no totally geodesic surfaces is a no-closed-orbits assumption for the action of $PSL(2,\mathbb{R})$ on the frame bundle of $M$. It is analogous to a unique ergodicity assumption for (one-dimensional) dynamical systems.  For uniquely ergodic dynamical systems on a compact space, a simple argument shows that the equidistribution ergodic theorem holds for \textit{all} space averages,  not just almost all \cite[Proposition 1.9]{pisanotes}.  We use a similar argument, together with Ratner's measure classification theorem, to prove that geodesic disks in $Gr_2((M,g_{hyp}))$ are becoming uniformly distributed at some rate uniform in the radii of the disks.  We then locally approximate the $g_{hyp}$-minimal surfaces $\Sigma_n$  by large totally geodesic disks to show that these surfaces are becoming uniformly distributed.  %Finally, we use the
%maps $f$ conjugating map $\Phi$ of Theorem \ref{mainintro} and the fact that the universal covers of $(M,g)$ and $(M,g_{hyp})$ are quasi-isometric to transfer information from $(M,g_{hyp})$ to $(M,g)$.  This last step closely follows arguments in (\cite{cmn}.)    

We expect that  the $\mu_n$  converge to the uniform measure \iffalse as long as their limit sets are converging to round circles and their areas are tending to infinity, \fi with no assumptions on $M$.    It seems likely, for example, that the probability measures corresponding to a sequence of minimal surfaces $\Sigma_n^{hyp}$ that injectively include to the surface subgroups that come from Hamenstadt's version \cite{ursulah} of the Kahn-Markovic construction are becoming uniformly distributed, but we do not verify this here.  

%Theorem \ref{quantitativeintro} also holds without the no-totally-geodesic-surfaces assumption. %(Our failure to prove this more general statement might stem more from the author's relative ignorance of homogeneous dynamics than the difficulty of the problem.)  
%It at least is true without this assumption for sequences $\Sigma_n$ as in that theorem such that all weak-$*$ limits of probability measures for the corresponding minimal surfaces in the constant curvature metric have full support in $Gr_2(M)$.  A proof can be given using the map $\tilde{f}_{\Sigma}$ constructed in the proof of Theorem \ref{uniqintro} together with arguments similar to those used to prove Theorem \ref{quantitativeintro}.                   

\subsection{Non-Existence of Foliations as in Theorem \ref{mainintro} } 

We now describe the construction of negatively curved metrics to which Theorem \ref{mainintro} cannot apply to produce a foliation.  It is based on the existence of quasi-Fuchsian manifolds $Q$ which contain several distinct embedded  minimal surfaces whose inclusions are homotopy equivalences (this contrasts with the almost-Fuchsian case, where it was shown in \cite{uhl} that there exists a unique such minimal surface.) 

We start out with a closed hyperbolic 3-manifold $M$ that contains an embedded totally geodesic surface. By passing to a finite cover which we also denote by $M$, we can make the totally geodesic surface $\Sigma$ have arbitrarily large normal injectivity radius. Taking the Fuchsian cover $F$  corresponding to $\Sigma$, which is homeomorphic to $\Sigma \times \mathbb{R}$, we modify the pulled-back metric while preserving negative curvature, so that we can cut out the middle of $F$ and glue in the middle of a quasi-Fuchsian $Q$ with multiple distinct minimal surfaces. To accomplish this, we also need to modify the metric on $Q$ near infinity, which we do using the fact that a quasi-Fuchsian metric on $\Sigma \times \mathbb{R}$, whatever disorderly behavior is happening in the middle, has a standard form near the ends.  Provided the normal injectivity radius of $\Sigma$ was taken large enough, the gluing can be performed inside $M$ itself.  This produces a negatively curved metric $g$ on $M$ for which there are multiple stable minimal surfaces isotopic to $\Sigma$, which is incompatible with the existence of a foliation as in Theorem \ref{mainintro}.   

We expect that this metric can be joined to the constant curvature metric through a smooth path of metrics with negative sectional curvature by performing the above construction on a smooth path in quasi-Fuchsian space joining $F$ to $Q$. This would show that the case $T<\infty $ of Theorem \ref{mainintro} actually occurs.   It would be good to find a more robust way of ruling out the existence of foliations as in Theorem \ref{mainintro}--- for instance, one that worked for all closed hyperbolic 3-manifolds.

\subsection{Stability for the Foliations of Theorem \ref{mainintro} } 

In the final section, we give an estimate for how fast the principal curvatures (of the projections to $M$) of the leaves of the foliations $\mathcal{F}_t$ are changing as the metrics $g_t$ vary.  The bound we obtain depends on the size of the principal curvatures of the (projections to $M$) of the leaves of the foliation compared to the ambient Ricci curvature, as well as bounds on  the $g_t$ and their derivatives in time.

\subsection{Related Work} 

We now discuss some results related to this paper.  Density and equidistribution theorems for minimal hypersurfaces produced by the Almgren-Pitts min-max theory have been obtained for generic metrics by Irie-Marques-Neves \cite{imn} and Marques-Neves-Song \cite{mns}.  Recently Song-Zhou \cite{songzhou} showed that for generic metrics sequences of minimal hypersurfaces can ``scar" along stable minimal hypersurfaces, for example the ones considered in this paper. The proofs of the above results are based on the Liokumovic-Marques-Neves Weyl law for the Almgren-Pitts volume spectrum \cite{lmn}.  Ambrozio-Montezuma \cite{ambmont} also proved equidistribution results, by a somewhat different approach,  for minimal surfaces in metrics on the round 3-sphere that are local maxima for the Simon-Smith width within their conformal class. In contrast to the minimal surfaces considered in this paper, the minimal surfaces of most of the results mentioned in this paragraph	 are embedded and one expects them in general not to be local minima for the area functional.   

Recent work of Calegari-Marques-Neves \cite{cmn} considered minimal surfaces corresponding to the Kahn-Markovic surface subgroups from a dynamical perspective.  Given a closed hyperbolic 3-manifold, they define a functional on Riemannian metrics on that 3-manifold with sectional curvature at most $-1$ based on a renormalized count of stable properly immersed minimal surfaces with limit sets close to circles, and show that the constant curvature hyperbolic metric uniquely minimizes this functional. The proof of the rigidity part of their result--- that the constant curvature metric \textit{uniquely} minimizes the counting functional--- uses the Ratner-Shah theorem mentioned earlier.  This paper was inspired by and draws substantially from their ideas, especially Sections \ref{applications} and \ref{stronger}.  

In Section \ref{mainsection}, we produce the leaves of the foliations of Theorem \ref{mainintro} by solving specific asymptotic Plateau problems in $\tilde{M}$, and arguing that the solutions are unique.  The asymptotic Plateau problem in $\mathbb{H}^n$ for suitable boundary data at infinity was solved by Anderson \cite{and2}, and in simply connected Riemannian manifolds bi-Lipschitz equivalent to a metric with pinched negative sectional curvature by Bangert and Lang \cite{banglang}. %Our need for quantitative control on solutions to the relevant asymptotic Plateau problems and to prove that they are unique prevents us from simply applying \cite{banglang}. 
Using their results might shorten our proof a little, but we prefer to construct the solutions to our asymptotic Plateau problems by hand. The main point for us is controlling the solutions to the relevant asymptotic Plateau problems as the metric varies and proving that the solutions are unique.     %Applying \cite{banglang} would mildly simplify some of our arguments, but we prefer to solve the asymptotic Plateau problem directly.  

%Our approach to solving the asymptotic Plateau problem has much greater specificity and is less technically sophisticated than \cite{banglang}, but it produces unique solutions.       

%As a step to proving Theorem \ref{quantitativeintro}, we show in Proposition \ref{balls} that  have full support in $Gr_2(M)$, provided that $M$ in its constant curvature metric has no proper totally geodesic surfaces.  In the constant curvature case it follows from Ratner's measure classification theorem (as we describe in Section \ref{stronger}) that provided $M$ has no proper totally geodesic surfaces, the only possible weak-$*$ limit of totally geodesic disks with radii tending to infinity is the Haar measure. 

By Theorem \ref{quantitativeintro} and Remark \ref{mozshah}, for any infinite sequence of closed leaves of our foliations with areas tending to infinity (these correspond to properly immersed totally geodesic surfaces in the hyperbolic metric) the sequence of probability measures $\mu_n'$ on $Gr_2(M)$ that they determine are weak-$*$ converging to the measure $\mu_g$ of Theorem \ref{quantitativeintro}.  One wonders what can be said about this measure in general.  Its regularity depends on the regularity of the conjugating map $\Phi$ in directions transverse to the leaves.   

\iffalse 
has the full support property (all weak-$*$ limits have full support.)
 
\begin{que} 
 What can be said about the possible weak-$*$ limits of the $\mu_n'$?    
\end{que}

\noindent In the constant curvature case, the only possible weak-$*$ limit is the uniform measure by \cite{mozesshah}. One could also ask about the possible weak-$*$ limits of the probability measures on $Gr_2(M)$ corresponding to metric disks with radii tending to infinity in (the projections to $M$ of) leaves of our foliation.  
In the constant curvature case it follows from Ratner's measure classification theorem (as we explain in Section \ref{stronger}) that provided $M$ has no properly immersed totally geodesic surfaces, the only possible weak-$*$ limit of totally geodesic disks with radii tending to infinity is the Haar measure. 

\fi

The ergodic theory of foliations with negatively curved leaves has been studied (\cite{alvarezsuspension}, \cite{walchuk},\cite{zimmerfoliations}.) In \cite{alvarezsuspension}, Alvarez considers certain foliations, transverse to the fibers of $\mathbb{CP}^1$ bundles over a negatively curved surface, that arise from actions of the fundamental group of the surface on $\mathbb{CP}^1$. He shows that there is a unique probability measure to which metric disks tangent to the leaves and with radii tending to infinity converge, and that this measure is singular with respect to other measures natural to the dynamics of the foliation unless the surface in the construction had constant negative curvature. It would be interesting to determine whether the story is similar for the foliations of this paper.

\section{Outline and Acknowledgements} 
In Section \ref{mainsection} we prove Theorem \ref{mainintro}.  In Section \ref{applications} we apply Theorem \ref{mainintro}  to prove density and uniqueness results for stable properly immersed minimal surfaces in $M$. In  Section \ref{stronger} we prove some quantitative versions of the density results of Section \ref{applications} under the assumption that $M$ has no proper totally geodesic surfaces in its hyperbolic metric. \iffalse In Section \ref{counterexample} we construct examples of negatively curved metrics to which Theorem \ref{mainintro} cannot apply to produce a foliation. \fi  In Section \ref{stabilitysect} we give an estimate for how fast the principal curvatures of the leaves of the foliations of Theorem \ref{mainintro} are changing as the metric varies.       
  
I would like to thank Fernando Al Assal, Clark Butler, Ilya Khayutin, Peter Sarnak, Andrea Seppi, Antoine Song, and Shmuel Weinberger for useful conversations and correspondence. I thank Alex Eskin for multiple helpful discussions related to Section \ref{stronger}.  I especially thank my advisor Fernando Coda Marques for his support and valuable suggestions, in particular related to Section \ref{stabilitysect}.

\section{Construction of the Foliations}  \label{mainsection} 

In this section we prove Theorem \ref{mainintro}.  Fix a closed hyperbolic 3-manifold M, and denote by $g_{hyp}$ the hyperbolic metric on $M$.  Let $\mathcal{P}$ be the set of totally geodesic planes in $\mathbb{H}^3$. By taking limit sets, there is a bijection between $\mathcal{P}$ and the set of round circles in $S^2 \cong \partial_\infty \mathbb{H}^3$. The lifts to $Gr_2(M)$ by their tangent planes of the projections of elements of $\mathcal{P}$ under the covering map are the leaves of a foliation of $Gr_2(M)$, which we denote by $\mathcal{F}$.  

%Let $\{g(t): t \in [0,1)\}$ be a smooth 1-parameter family of metrics on $M^3$, with $g(0)=g_{hyp}$.  After a choice of basepoints, there is a natural identification between the universal covers $(\tilde{M},\tilde{g}(t))$ by lifting the identity maps between the $(M,g(t))$.  Fix such an identification, and let $\mathcal{P}_t$ be the set of embedded planes in $(\tilde{M},\tilde{g}(t))$ that are images of the totally geodesic planes in $\mathcal{P}$ under the identification between $(\tilde{M},\tilde{g}(t))$ and $(\tilde{M},\tilde{g}(0))\cong \mathbb{H}^3$.  Since the identification between $(\tilde{M},\tilde{g}(t))$ and 

Let $g$ be a metric on $M$.  Then there is an identification between the universal cover $(\tilde{M},\tilde{g})$ of $(M,g)$ and the universal cover $\mathbb{H}^3$ of $(M,g_{hyp})$, which is well-defined up to composing with covering transformations of $\mathbb{H}^3$. Since elements of the set $\mathcal{P}$ are invariant under covering transformations, taking the images of elements of $\mathcal{P}$ under such an identification gives a well-defined set of embedded planes in $(\tilde{M},\tilde{g})$, which we denote by $\mathcal{P}_{\tilde{g}}$.

\begin{defn} Let $g$ be a metric on $M$ with negative sectional curvature.  Consider the universal cover $(\tilde{M},\tilde{g})$ with the metric induced by $g$.  We say that an embedded surface $\Sigma$ in $(\tilde{M},\tilde{g})$ is $\epsilon$\textit{-subordinate} if it satisfies the following for every $p \in \Sigma$.  Let $\nu$ be the unit normal vector to $\Sigma$ at $p$ and let $A$ be the second fundamental form of $\Sigma$.  Then
	\begin{equation} \label{convex} 
	|A(p)|^2 < |Ric(\nu, \nu)| - \epsilon.   
	\end{equation} 

We say that $g$ is in $\Omega_\epsilon$ if it has negative sectional curvature and there is some $\epsilon>0$ such that for every $P \in \mathcal{P}_{\tilde{g}}$, there is a properly embedded $\epsilon$-subordinate minimal plane in $(\tilde{M},\tilde{g})$ at finite Hausdorff distance from $P$. 

%, and that $g$ is in $\Omega$ if it is in $\Omega_\epsilon$ for some $\epsilon$. 

%We say that $g$ is in $\Omega_0$ if it can be joined to $g_{hyp}$ by a smooth path of metrics in $\Omega$.  We will show that for metrics in $\Omega_0$, there is actually a unique proper minimal plane at finite distance from every element of $\mathcal{P}_{\tilde{g}}$. 

\end{defn}  

\begin{rem}
For a plane $P \in \mathcal{P}$, it will either be the case that the (lift to $Gr_2(M)$ of the) projection of $P$ to $(M,g_{hyp})$ is dense in $Gr_2(M)$, or closes up to a properly immersed surface (\cite{ratner}, \cite{shah}.) Take a plane $P$ with dense projection to $Gr_2(M)$. Then, for $g$ a negatively curved metric on $M$, if there is an $\epsilon$-subordinate minimal plane $S$ in $(\tilde{M},\tilde{g})$ at finite Hausdorff distance from $P$ for some $\epsilon>0$, then $g \in \Omega_{\epsilon'}$ for any $\epsilon' < \epsilon$. That is, it suffices to check Equation (\ref{convex}) on a single embedded plane corresponding to an element of $\mathcal{P}$ with dense projection to verify membership in $\Omega_{\epsilon'}$ for $\epsilon'<\epsilon$.  This can be seen by approximating any $P' \in \mathcal{P}$ by orbits of $P$ under covering transformations, taking the corresponding sequence of minimal surfaces in $(\tilde{M},\tilde{g})$, and passing to a smooth subsequential limit as in the proof of Theorem \ref{main} below to produce a minimal plane satisfying Equation (\ref{convex}) for any $\epsilon'<\epsilon$.     
\end{rem}

%\begin{rem} 
%Clearly the hyperbolic metric is contained in $\Omega_{\epsilon}$ for $0<\epsilon<2$.  Bounds for the size of the path component of the hyperbolic metric in $\Omega_\epsilon$ follow from the final section.    

%It follows from work of Gromov \cite{groapp} that for every $0<\epsilon<1$,  metrics with sectional curvatures pinched sufficiently close to -1 are in the path component of the hyperbolic metric in $\Omega_{\epsilon}$. 
%\end{rem} 

Presumably $\Omega_{\epsilon}$ for $\epsilon$ close to zero contains more metrics than just those with sectional curvatures extremely close to -1.  It would be nice to have a better understanding of which metrics are contained in $\Omega_\epsilon$ for $\epsilon$ small.  Are all metrics that can be smoothly joined to the hyperbolic metric through metrics with sectional curvatures pinched between $-1$ and $-3/2$  contained in $\Omega_{\frac{1}{1000}}$, for example?  (Conceivably, the answer could depend on $M$.)   % We expect that getting any kind of precise quantitative description of the boundary of $\Omega_{\epsilon}$ in the space of all negatively curved metrics would be difficult, however, since the boundary of the space of almost-Fuchsian metrics in quasi-Fuchsian space seems to be poorly understood.  

%Results that relate different measures of the complexity of an almost-Fuchsian manifold, like the Hausdorff dimension of the limit set, to the $L_\infty$ norm of the principal curvatures of its unique central minimal surface.   

We now prove the main theorem of the section.  It does not quite imply Theorem \ref{mainintro}, but we will explain at the end how the proof can be modified to give a proof of Theorem \ref{mainintro}

\begin{thm} \label{main}  
Let $\{g_t: t \in [0,1] \}$ be a smooth 1-parameter family of metrics on $M^3$, with $g_0=g_{hyp}$ and $g_t \in \Omega_\epsilon$  for some fixed $\epsilon>0$ and all $t$. Then there exists a constant $C$ depending only on the family of metrics such that for all $P  \in  \mathcal{P}_{\tilde{g}_t}$ there exists a properly embedded minimal plane $S_t$ in the universal cover $(\tilde{M},\tilde{g}_t)$ at a Hausdorff distance from $P$ of at most $C$, and that has the following properties: 
\begin{itemize} 
	\item $S_t$ is the unique properly embedded minimal plane at finite Hausdorff distance from $P$
	
	\item $S_t$ is absolutely minimizing.  
	
	\item The lifts of the $S_t$ to $Gr_2(\tilde{M})$ by their tangent planes are the leaves of a foliation $\tilde{\mathcal{F}_t}$ of $Gr_2(\tilde{M})$. 
	\end{itemize} 
 
  \noindent The diffeomorphisms of $Gr_2(\tilde{M})$ induced by covering transformations of $\tilde{M}$  send leaves to leaves, and $\tilde{\mathcal{F}_t}$ thus descends to a foliation $\mathcal{F}_t$ of $Gr_2(M)$.  Moreover, $\mathcal{F}_t$ and $\mathcal{F}$ are conjugate, in that  there is a homeomorphism 
\[
\Phi: Gr_2((M,g_{hyp})) \rightarrow Gr_2((M,g_t))
\]
 that maps leaves of $\mathcal{F}$ to leaves of $\mathcal{F}_t$. 
\end{thm} 

We say in this paper that a minimal surface is \textit{absolutely minimizing} if, for every piecewise-differentiable closed curve on the surface that bounds a disk $D$ on the surface, the area of $D$ is less than or equal to that of any other smoothly embedded disk in the ambient space bounding $\partial D$.  

A 1-parameter family of Riemannian metrics as in the theorem gives a map 
\[
M \times [0,1] \rightarrow Sym^2(T^*M).  
\]
We say that the family of metrics is smooth if this map is smooth.

%\begin{rem} We point out that we do not know how to prove the corresponding result for metrics in $\Omega- \Omega_0$, which seem likely to exist.
%	\end{rem} 
%We start out with a lemma.  
%\begin{lem} \label{compactness} 
%	Let $(M,g_n)$ be a sequence of metrics with negative sectional curvature that smoothly converges to a metric $(M,g)$ with negative sectional curvature.   
	
%	\end{lem} 

%We now give the proof of Theorem \ref{main}.  
%\begin{proof} 
	
	Let $\{g_t: t \in [0,1]\}$ be a smooth 1-parameter family of metrics on $M^3$, with $g_0=g_{hyp}$ and $g_t \in \Omega_\epsilon$ for all $t$. We will prove Theorem \ref{main} by a finite induction.  Suppose that $g_{t_0}$ satisfies the conclusion of the theorem for some $t_0$.  In the proof, we will use the following two properties of the $S_{t}$ at $t=t_0$, the first of which we will assume at $t_0$ and verify in the inductive step and the second of which follows from the existence of the conjugating map $\Phi$ in the theorem and standard elliptic PDE theory.

	   \begin{itemize} 
	    \item [\textbf{Property 1}] Suppose lifts of $S_t$ and $S_t'$ are leaves of $\tilde{\mathcal{F}}_t$ that correspond to totally geodesic planes $S$ and $S'$ in $\mathcal{P}$.  If $S$ and $S'$ have disjoint boundary circles at infinity, then $S_t$ and $S_t'$ are disjoint.  
	    \item	[\textbf{Property 2}] If $S_n$ is a sequence of totally geodesic planes that converges to $S$ on compact subsets, then the corresponding sequence of minimal planes in $(\tilde{M},\tilde{g_{t}})$ smoothly converges, uniformly on compact subsets, to the minimal plane in $(\tilde{M},\tilde{g}_{t})$ corresponding to $S$.    
	 
\end{itemize}

	% We will show that there is some $\overline{\epsilon}$ independent of $t_0$ so that for all  $t\in (t_0,t_0 + \overline{\epsilon})$, $g(t)$  satisfies the conclusion of the theorem, as well as the further assumptions of the previous paragraph.  
	
\subsection{Outline}  We carry out the induction in three steps.  First we construct the minimal planes $S_t$ in the universal cover as limits of solutions to Plateau problems for a sequence of circles going off to infinity, roughly following  the approach introduced by \cite{and2} to solving the asymptotic Plateau problem in negative curvature.  Next, using the existence of mean-convex tubular neighborhoods of the $S_{t_0}$ guaranteed by $g_{t_0}$'s membership in $\Omega_\epsilon$, we prove that the $S_t$ are unique.  Finally, based on the strong restrictions on how minimal surfaces can intersect in three dimensions, we prove that the lifts of the $S_t$ to $Gr_2(\tilde{M})$ by their tangent planes give a foliation of $Gr_2(\tilde{M})$.

	\subsection{Mean-Convex Neighborhoods}  %\textbf{Construction of the $S_t$}

	%First, note that for any $S_{t_0}$ as in the statement of the theorem there is some $\epsilon'$ independent of $S_{t_0}$ and $t_0$ such that the normal exponential map $S_{t_0} \times (-\epsilon',\epsilon') \rightarrow (\tilde{M},\tilde{g(t_0})$ is a diffeomorphism onto its image.  This is because we have uniform bounds on the geometry of the $(\tilde{M},\tilde{g(t_0)})$ and the second fundamental forms of the $S_{t_0}$ by Equation \ref{convex}.  It therefore makes sense to talk about the distance $r$ parallel surfaces to $S_{t_0}$ for $|r|<\epsilon'$.  

	 Let $S_t$ be an $\epsilon$-subordinate properly embedded minimal disk in $(\tilde{M}, \tilde{g}_t)$ as in the statement of the theorem.  Along any normal geodesic ray $\gamma$ from $S_t$ parametrized by arc-length and within the normal injectivity radius, the signed mean curvatures $m$ of the parallel surfaces satisfy the following equation:   
	 \begin{equation} \label{riccati} 
	 m'((\gamma(s)) = -|A(\gamma(s))|^2 -\text{Ric} (\dot{\gamma}(s) ,\dot{\gamma}(s)),  
	 \end{equation} 
	 
	 \noindent where $A(\gamma(s))$ is the second fundamental form of the  signed-distance-$s$ parallel surface at $\gamma(s)$.  %, and where $|s|$ is less than the normal injectivity radius of $S_t$. 
	 This can be obtained by taking the trace of Equation (2) in Proposition 3.2.11 of \cite{petersen}. (See also Gromov's survey article \cite{gromovsignandgeometricmeaning} which develops some of the themes of modern geometry through this and related ``tube-formulas.")  %[Reality check with totally geodesic plane in $H^3$, tanhr]  
	 The next lemma will be used at several points below.    
	 
	 \begin{lem} \label{lemconvex} 
	 	 There is some $\xi$ depending only on $\epsilon$ and the family $g_t$ such that the parallel signed distance-$r$ surfaces of the $S_{t}$ have mean curvature greater than $\frac{\epsilon}{2} r$ if $0<r<\xi$ and less than $\frac{\epsilon}{2} r$ if $-\xi < r<0$. 
	 	\end{lem} 
	 	\begin{proof}

	 		 First note that we have uniform bounds on the $L^\infty$ norm of the Ricci curvature tensor over all $g_t$.  Since $S_{t}$ is $\epsilon$-subordinate, there is thus by (\ref{convex}) a uniform bound on the magnitude of its second fundamental form, depending only on $\epsilon$ and bounds on the $g_t$.   The upper bound on the magnitude of the second fundamental form and the uniform bounds on $\tilde{g}_t$ and its derivatives imply a lower bound on the normal injectivity radius of $S_{t}$ uniform over all $S_t$--- i.e., a lower bound for an $\epsilon'$ such that the normal exponential map on the normal bundle to $S_t$ is injective restricted to $S_t \times (-\epsilon',\epsilon')$.    
	 		 
	 	Now suppose the statement of the lemma were false, and let $\{t_n\}$, $\{r_n\}$, and $\{x_n\}$ be sequences of times, signed-distances, and points on $S_{t_n}$ such that:  
	 	
	 	\begin{itemize} 
	 	\item the distance-$r_n$ surface to $S_{t_n}$ has mean curvature less than $\frac{\epsilon}{2} r_n$ if $r_n>0$ or greater than $\frac{\epsilon}{2}r_n$ if $r_n<0$ at the point that normally projects to $x_n$
	 	\item $|r_n| \rightarrow 0$ 
	 	\item $\{t_n\}$ converges to some time $t$ (where we've passed to a subsequence if necessary.)      	
	 		\end{itemize}
	 		
	 	\noindent Let $K$ be a compact set containing a fundamental domain for the action of $\pi_1(M)$ on $\tilde{M}$, and for each $x_n$, let $\gamma_n$ be a covering transformation of $\tilde{M}$ such that $\gamma_n \cdot x_n \in K$. By passing to a subsequence we can assume that $\gamma_n \cdot x_n$ converges to $x$.  By the uniform bound on the second fundamental forms of the $S_{t_n}$, we can pass to a subsequence of the $\gamma_n \cdot S_{t_n}$ that graphically converges (and thus, by standard elliptic PDE theory, smoothly converges) in a neighborhood of $x$ to a $\tilde{g}_t$-minimal disk $D$ containing $x$.  Since this disk inherits the property of being $\epsilon$-subordinate from the $\gamma_n \cdot S_{t_n}$ of which it was a smooth limit, Equation (\ref{riccati}) implies that the derivative of the signed mean curvatures of the parallel distance-$r$ surfaces to $D$ at $r=0$ is greater than $\epsilon$ at every point in the interior of $D$. 
	 	
	 	 By passing to a subsequence, we can assume that all of the $r_n$ are either positive or negative; the argument is the same in both cases so assume that all are positive. By the mean value theorem there is a sequence of $r_n'$ such that the derivative of the mean curvature of the parallel distance-$r$ surfaces to $\gamma_n \cdot S_{t_n}$ along the normal geodesic to $\gamma_n \cdot x_n$ is less than $\epsilon/2$  at the distance-$r_n'$ surface, where $0<r_n'  < r_n$. Since neighborhoods of $\gamma_n \cdot x_n$ in  $\gamma_n \cdot S_{t_n}$ are smoothly converging to $D$, their parallel distance-$r$ surfaces are smoothly converging to those of $D$.  This implies that the derivative of the mean curvature of the parallel distance-$r$ surfaces to $D$ along the normal geodesic at $x$ is less than or equal to $\epsilon/2$ at $r=0$, which contradicts the previous paragraph.

	% This follows from Equation (\ref{riccati}.) To see this,=  By standard elliptic PDE theory it also implies uniform bounds on all derivatives of the principal curvatures of $S_t$.  This implies bounds on the sizes of the derivatives of the principal curvatures of parallel distance $r$ surfaces to $S_t$ along geodesics normal to $S_t$, for  $|r|$ less than some $r'$ that only depends on bounds on the geometry of the $g(t)$ and $\epsilon$. Therefore, since the $\epsilon$-convexity of $S_t$ implies that the right-hand side of (\ref{riccati}) is greater than $\epsilon$ at $s=0$, there exists $\xi$, depending only on $\epsilon$ and bounds on the geometry of $g(t)$, such that it is also greater than $\frac{\epsilon}{2} s$ for $|s|<\xi$, and for these $s$ the parallel distance-$s$ surfaces will be strictly mean-convex. 
	 
	 %and thus a uniform lower bound on its injectivity radius
	\end{proof} 
	
	  Let $S_{t_0}^r$ be the parallel surface at signed-distance $r$ from $S_{t_0}$.  Then by the previous lemma for $\delta$ sufficiently small (and independent of $t_0$), $S_{t_0}^r$ will remain mean-convex when considered as a surface inside $(\tilde{M},\tilde{g}_t)$, for $t \in (t_0, t_0 + \delta)$ and $\frac{\xi}{4} < |r|< \xi$. For $t \in (t_0, t_0 + \delta)$, we now construct the $S_{t}$. At several junctures below, we will put further restrictions on the size of $\delta$ that only depend on $\epsilon$ and the family of metrics.  
	 
	 \subsection{Controlled Solutions to Plateau Problems}
	 Fix a point $p$ on $S_{t_0}$ and let $B(s)$ be the metric disk in $S_{t_0}$ of radius $s$ centered at $p$, where $S_{t_0}$ has the metric induced from $\tilde{g}_{t_0}$. Because $S_{t_0}$ is minimal and $\tilde{g}_{t_0}$ has negative sectional curvature, $S_{t_0}$ has negative curvature and the exponential map  $T_pS_{t_0} \rightarrow S_{t_0}$ is a diffeomorphism. The boundaries $\partial B(s)$ are therefore embedded circles, and so % which are contained in uniformly $\tilde{g}_{t_0}$-convex balls which, if $\delta$ was taken small enough, will also be convex in $\tilde{g}_{t}$. Therefore, by \cite{meeksyau} 
	 we can solve the Plateau problem for  $\partial B(s)$ in $(\tilde{M},\tilde{g}_{t})$ to find an embedded $\tilde{g}_t$- minimal disk $D(s)$ that bounds  $\partial B(s)$, such that every other embedded disk bounding  $\partial B(s)$ has area greater than or equal to that of $D(s)$ \cite[Chapter 4]{coldmini}.

	 \begin{lem}  \label{trapping} 
	 	
	 	$D(s)$ is contained in the region bounded by $S_{t_0}^r$ and $S_{t_0}^{-r}$ for $\xi/4<r<\xi$.
	 \end{lem} 
	 
	 We first prove another lemma. Let $S$ be the geodesic plane in $\mathbb{H}^3$ that corresponds to $S_{t_0}$. The circles in $S^2$ parallel to the boundary at infinity of $S$ form a foliation of $\partial_{\infty}\mathbb{H}^3 \cong S^2$ minus two points.  Let $\{S(x): x \in \mathbb{R} \}$ be the foliation of $\mathbb{H}^3$ by totally geodesic planes whose limit sets are the circles in $S^2$  parallel to $\partial_\infty S$, parametrized so that $S(0)=S$, and let $S_{t_0}(x)$ be the corresponding minimal planes in $(\tilde{M},\tilde{g}_{t_0})$.  
	 
	 \begin{lem} \label{vertfoliation} 
	 	 The $S_{t_0}(x)$ are the leaves of a foliation of $\tilde{M}$.
	 	\end{lem} 
	 	
	 	\begin{proof} 
	 		 
	 		 First of all by Property 1 the $S_{t_0}(x)$ are disjoint. %We claim that for any fixed $x_0$ the minimal planes $S_{t_0}(x)$ in $(\tilde{M},\tilde{g_{t_0}})$ corresponding to the $S(x)$ converge smoothly to $S_{t_0}(x_0)$ from above and below as $x\rightarrow^+ x_0$ and $x \rightarrow^- x_0$ (having chosen orientations so that above and below make sense.)  
	 		 Suppose for contradiction that there is some point $q$ that is not contained in any $S_{t_0}(x)$.  Let $x_q^+$ be the infimum of the set of all $x$ such that $S_{t_0}(x)$ is above $q$.  %, and $x_q^-$ be the supremum of the set of all $x$ such that $S_{t_0}(x)$ is below $q$.  
	 		 This set is nonempty since the $S_{t_0}(x)$ are at uniformly bounded Hausdorff distance from the $S(x)$ considered as subspaces of $(\tilde{M}, \tilde{g}_t)$, so $x_q^{+}$ is well-defined.  If $S_{t_0}(x_q^+)$ did not contain $q$, then since the $S_{t_0}(x)$ vary smoothly by Property 2 above, it must be above $q$ and we could therefore produce $S_{t_0}(x)$ above $q$ with $x< x_q^+$ for a contradiction.   % if $q$ were not contained in $S_{t_0}(x_q^+)$, and similarly if $q$ were not contained in  $S_{t_0}(x_q^-)$, for a contradiction.  
	 		 It follows that $q$ is contained in $S_{t_0}(x_q^+)$.  The existence of local product charts follows from Property 2, and so the $S_{t_0}(x)$ are the leaves of a foliation.

	 	\end{proof} 
	 	
	 	\noindent We now give the proof of Lemma \ref{trapping}.

	 \begin{proof} 	 
	                                                                                               Suppose for contradiction that $D(s)$ is not contained in the region bounded by $S_{t_0}^r$ and $S_{t_0}^{-r}$. Then either $D(s)$ has points above $S_{t_0}^{r}$ or below $S_{t_0}^{-r}$--- assume that the first is true since the proof in the second case is the same. Note that the parallel signed-distance-$r$ surfaces $S_{t_0}^r(x) $ for the $S_{t_0}(x) $ in $(\tilde{M},\tilde{g_{t_0}})$ are mean-convex in $(\tilde{M},\tilde{g_t})$ since $\frac{\xi}{4} < r< \xi$ and are the leaves of a foliation of $\tilde{M}$ for fixed $r$, since the  $S_{t_0}(x)$ are the leaves of a foliation of $\tilde{M}$ by the previous lemma.  It follows that the set of $x>0$ such that $S_{t_0}^r(x)$ intersects $D(s)$ is non-empty. 
	                                                                                               
	                                                                                               Furthermore, since the $S_{t_0}^r(x)$ are at uniformly bounded Hausdorff distance from the $S(x)$ considered inside $\{\tilde{M}, \tilde{g}_t \}$, for $x$ sufficiently large the intersection of $S_{t_0}^r(x)$ and $D(s)$ will be empty. % and the natural map from $(\tilde{M},\tilde{g}_0)\cong \mathbb{H}^3$ to  $(\tilde{M},\tilde{g}_t)$ is a quasi-isometry.  
	                                                                                      Let $x'$ be the largest $x$ such that  $S_{t_0}^r(x)$  intersects $D(s)$. Then since $D(s)$ is contained on one side of $S_{t_0}^r(x')$ except at the non-empty set of points where they intersect, the mean convexity of $S_{t_0}^r(x')$ gives a contradiction.

	 \end{proof} 
	 
	% We now describe how to take a limit of the disks $D(s)$ to produce the surfaces $S_t$.
	
	The next lemma follows from Schoen's curvature estimate for stable minimal surfaces \cite[Theorem 3]{schoencurvestimate}.
	 \begin{lem} \label{schoen} 
	Let $\rho>0$ be less than the injectivity radius of any $(M,g_t)$, and let $x$ be a point on a stable $\tilde{g}_t$-minimal embedded disk $D \subset \tilde{M}$ at a $\tilde{g}_t$-distance of at least  $d$ from the boundary of $D$. Then there is some constant $C$  such that the $L^\infty$-norm of the second fundamental form of $D$ at p is bounded above by $C$.  The constant $C$ depends only on $\rho$, $d$, and bounds on the $g_t$ (specifically, bounds on the $L^\infty$ norms  of the curvature tensor and its covariant derivative.)

	\end{lem} 
	
%	\begin{proof} 
	   %The bounds on the second fundamental form of a stable minimal surface are in terms of bounds on the geometry of $g_t$ (specifically, bounds on the $L^\infty$ norms  of the curvature tensor and its covariant derivative) and the distance from 
	 
	%  so we can choose a uniform $C$ that works for every metric in our family.  
%	 \end{proof} 

	 We require two more lemmas to get the control on the $D(s)$ we will need to pass to a limit.

	 \begin{lem} \label{normalvector} 
	 	Let $\eta>0$, $C$, and the family $g_t$ be given. Then there exists $\eta'>0$ such that for any $t$ and any two embedded surfaces $S_1$ and $S_2$ in  $(\tilde{M}, \tilde{g}_t)$ with second fundamental forms bounded above by $C$ in magnitude, the following holds.  Suppose that there are points $x_1 \in S_1$ and $x_2 \in S_2$ such that $d_{\tilde{g}_t}(x_1, x_2) < \eta'$, but the distance between the unit normal vectors to $S_1$ and $S_2$ at $x_1$ and $x_2$ respectively in the unit tangent bundle to $\tilde{M}$ is at least $\eta$.  Then $S_1$ and $S_2$ intersect.  
	 	\end{lem} 
	 	
	 	\begin{proof} 
	 			Assume that the statement of the lemma fails for some $\eta>0$, and let $\{S_i^n: i =1,2\}$ be a sequence of pairwise disjoint surfaces in $(\tilde{M}, \tilde{g}_{t_n})$ as above with points $x_i^n$ such that $d(x_1^n,x_2^n) \rightarrow 0$ and the distance between the respective normal vectors at $x_1^n$ and $x_2^n$ is at least $\eta$ for all $n$. 
	 			
	 			Fix some $\rho>0$. %smaller than the injectivity radius of all the $g_t$.  
	 			For each $n$, identify by the exponential map the ball of radius $\rho$ at $x_1^n$ with the ball of that radius in the tangent space to $x_1^n$, and radially dilate the pulled back metric so that the distance between the origin and dilated image of the point corresponding to $x_{2}^n$ in the dilated metric is 1. %(that is to say, dilate by a factor of $\frac{1}{d(x_1^n,x_2^n)}$.) 
	 			 Call the dilated metric $h_n$. For what follows, we fix isometric identifications of the tangent spaces $T_{x_1^n} S_1^n$ in the inner product induced by $\tilde{g}_{t_n}$, for the purpose of taking limits.  
	 			
	 			   Because $d(x_1^n,x_2^n) \rightarrow 0$ and the $\tilde{g}_{t}$ are in a smooth family over $[0,1]$, the $h_n$-balls of any given radius centered at the origin are smoothly converging to Euclidean balls of that radius.  Moreover, since we have a uniform bound on the second fundamental forms of the $S_i^n$, the intersections of their pre-images with the $h_n$-balls of any given radius centered at the origin are uniformly $C^1$-converging to planes, up to taking subsequences.  
	 			
	 			In the case of the $S_1^n$, this plane will simply be a subsequential limit of the tangent planes to $S_1^n$ at $x_1^n$.  Let $\Pi_n$ be the parallel transport of the tangent plane to $S_2^n$ at $x_2^n$ along the geodesic joining $x_1^n$ to $x_2^n$.  Then in the case of $S_2^n$, the subsequential limit plane will be a translated copy of a subsequential limit of the $\Pi_n$. The fact that for all $n$ the normals at $x_1^n$ and $x_2^n$ are at a distance of at least $\eta$ implies that these two planes cannot be parallel. This means that $S_1^n$ and $S_2^n$ will have to intersect for some large $n$, a contradiction.

	 		\end{proof}

	 	We now truncate the $D(s)$ to obtain disks better suited to taking limits.  Since $B(s)$ is a disk bounding $\partial B(s)$ that intersects $\gamma$ exactly once, we know that $D(s)$ intersects $\gamma$ at least once, at some point $p_s$ on $\gamma$ between $\gamma(-r)$ and $\gamma(r)$ by Lemma \ref{trapping}. Let $\sigma(s)$ be the smallest number $\sigma$ such that the $\tilde{g}_t$-ball $B_{p_s}(\sigma)$ of radius $\sigma$ centered at $p_s$ intersects $\partial D(s)$. Since $\partial D(s)$ is the boundary of an intrinsic metric disk in $S_{t_0}$, and $S_{t_0}$ is properly embedded, it must be the case that $\sigma(s) \rightarrow \infty$ as $s \rightarrow \infty$.

	 	For generic $\sigma < \sigma(s) $, the boundary of the $\tilde{g}_t$-ball $B_{p_s}(\sigma)$ of radius $\sigma$ centered at $p_s$ will intersect $D(s)$ in a union of circles by Sard's theorem. Note that the minimality of $D(s)$ implies that all connected components of $B_{p_s}(\sigma) \cap D(s)$ are disks.  If there were an annuli, then its interior component in $D(s)$ would be a minimal disk $D'$ with boundary on the boundary of $B_{p_s}(\sigma)$.  Taking the largest $\sigma'$ so that $B_{p_s}(\sigma')$ intersected $D'$, the fact that metric spheres are mean-convex in negative sectional curvature gives a contradiction.

	 	Now choose $\sigma$ in the interval $(\sigma(s) - 1 - \frac{1}{s} , \sigma(s) - 1)$ so that  $\partial B_{p_s}(\sigma) \cap D(s)$ is a union of circles, and let $D(s)'$ be the connected component of $B_{p_s}(\sigma) \cap D(s)$ containing $p_s$, which by the above paragraph is a disk.  Since we took $\sigma$ in the above interval, Lemma \ref{schoen} applies to give an upper bound on the absolute values of the principal curvatures of the $D(s)'$.

	 	% Let $D'(s)$ be the set of points on $D(s)$ at $\tilde{g}_t$-distance at least 1 from $\partial D(s)$. Let $\gamma$ be the geodesic normal to $S_{t_0}$ at $p$ in $(\tilde{M},\tilde{g}_{t_0})$, which intersects $S_{t_0}(\pm r)$ at $\gamma(\pm r)$.  As long as $s$ is large enough, $p_s$ will be contained in $D'(s)$. Let $D_0'(s)$ be the connected component of $p_s$ containing the identity.  
	 
	 \begin{lem} \label{nofolding} 
	 There exists $\delta'$, depending only on $\epsilon$ and the family of metrics, such that as long as $\delta$ was chosen less than $\delta'$, the nearest-point projection of $D(s)'$ to $S_{t_0}$ is well-defined and a diffeomorphism onto its image.       
	 
	 %The constants $\overline{\epsilon}$ and $\epsilon''$ can be chosen independent of $t_0$.  
	 \end{lem} 
 
 \noindent In what follows, if we refer to a quantity as $O(x)$, we mean that it tends to zero as $x$ tends to zero.  
	 
	 \begin{proof}

	  Fix a point $q$ on $D(s)'$. By Lemma \ref{trapping}, every point on $D(s)'$ is at a distance of at most $\xi/2$ from a point on $S_{t_0}$.  For $r\in (\frac{\xi}{4},\xi)$ the  normal exponential map for $S_{t_0}$ defines a diffeomorphism from $S_{t_0} \times (-r,r)$ to the $r$-neighborhood of $S_{t_0}$, so let $q'$ be the point on $S_{t_0}$ such that the normal geodesic to $S_{t_0}$ at $q'$--- call it $\phi$--- passes through $q$. Let $q_{\pm r}'$ be the points in $S_{t_0}^{\pm r} \cap \phi$.  	%Let  $\hat{S}_{t_0}$, $\hat{S}_{t_0}^{\pm r}$ and $\hat{D}(s)'$ be lifts of the respective surfaces to the unit tangent bundle of $\tilde{M}$ by their normal vectors.  

	 The distances between the unit normal vectors to $D(s)'$ and $S_{t_0}^{\pm r}$ at $q$, $q_{r}'$ and $q_{-r}'$ in the appropriate orientations are pairwise $O(\xi)$ by Lemma \ref{normalvector}.  Lemma \ref{normalvector} applies because  the $S_{t_0}^{\pm r}$ have principal curvatures bounded above in absolute value and $D(s)'$ has principal curvatures bounded above in absolute value by Lemma \ref{schoen}. Since the  normal vectors to the $S_{t_0}^{\pm r}$ are $ O(\xi)$-close to those of $S_{t_0}$ at points that correspond under normal projection, it follows that the normal vectors to $D(s)'$ and $S_{t_0}$ at $q$ and $q'$ are $O(\xi)$-close. This implies that, provided $\xi(=O(\delta))$ and $\delta$ were taken small enough, that the tangent vector to $\phi$ at q is very close to being perpendicular to $D(s)'$.  Therefore, for points $q''$ on $S_{t_0}$ in a small neighborhood $U$ of $q'$, the geodesic normal to $S_{t_0}$ at $q''$ will intersect $D(s)'$ nearly perpendicularly at some point close to $q'$.  It follows that the normal exponential map defines a diffeomorphism from $U$ to a neighborhood of $q$, so the normal projection map from $D(s)'$ to $S_t$ is a local diffeomorphism at all points of $D(s)'$ including points on its boundary. The normal projection from $D(s)'$ to $S_{t_0}$ is therefore a proper local diffeomorphism so, since $D(s)'$ is a disk, it must be a diffeomorphism.

	 	\end{proof}

\subsection{Construction of Leaves} \label{construction} 
	 	
	 	We now construct $S_t$ as a limit of the $D(s)'$.  Let  $B_n$ be the metric disk in $S_{t_0}$ with radius $n$ and center $p$ in the metric on $S_{t_0}$ induced by $\tilde{g_t}$.  If $s$ is sufficiently large, the previous lemma tells us that the normal exponential map diffeomorphically maps each $B_n$ onto a region $B_n(s)$ in $D(s)'$.  $B_n(s)$ is therefore a graph over $B_n$ in normal exponential coordinates for $S_t$.  Since the $D(s)'$ have uniformly bounded principal curvatures, we can pass to a subsequence of the $B_n(s)$ for a sequence of $s \rightarrow \infty $ that $C^1$-converges (and thus by standard elliptic PDE theory smoothly converges) over compact subsets of $B_n$.  Doing this for each $B_n$ and taking a diagonal subsequence we obtain an embedded minimal surface which we call $S_t$.  Since $S_t$ is a smooth limit of the $D(s)'$, the normal projection map defines a diffeomorphism from $S_t$ to $S_{t_0}$, and $S_t$ is a smooth properly embedded plane.  The surface $S_t$ inherits the property of being absolutely minimizing from the $D(s)'$ of which it was a smooth limit.

\subsection{Uniqueness} \label{uniqsection}  
	 	
	 	 We now check that the uniqueness conditions in Theorem \ref{main} are met. Since $S_t$ is at a Hausdorff distance from $S_{t_0}$ bounded by $\xi$,  $S_t$ is at finite Hausdorff distance from some element $S$ of $\mathcal{P}_{\tilde{g}_t}$, since this is true for $S_{t_0}$.  Let $S_t'$ be the properly embedded $\epsilon$-subordinate $\tilde{g}_t$-minimal disk at finite Hausdorff distance from $S$, guaranteed by $g_t$'s membership in $\Omega_\epsilon$.  We will show that $S_t'=S_t$.

	 	\begin{lem} \label{lim} 
	 	 $S_t'$ is contained in the region bounded by $S_{t_0}^r$ and $S_{t_0}^{-r}$ for $r \in (\frac{\xi}{4},\xi)$.

	 	 \end{lem} 
	 	 
	 	 \begin{proof} 
	 	  The proof is similar to that of Lemma \ref{trapping} but with an extra step because $S_t'$ is not compact.  Let $S_{t_0}(x)$ be the foliation given by Lemma \ref{vertfoliation}, and let $S_{t_0}^r(x)$ be the signed-distance-$r$ parallel surfaces to the $S_{t_0}(x)$, where $\frac{\xi}{4} < |r|< \xi$ and these surfaces are all mean-convex.  Now assume that $S_t'$ has points that are not contained in the region bounded by  $S_{t_0}^r$ and $S_{t_0}^{-r}$--- suppose for contradiction that it has points above $S_{t_0}^r$.  Let $x_{\max}$ be supremal over all $x$ such that   $S_{t_0}^r(x)$ intersects $S_t'$.  The number $x_{\max}$ is finite because $S_t'$ is at finite Hausdorff distance from $S$ and each $S_{t_0}^r(x)$ is at finite and uniformly bounded Hausdorff distance from the corresponding element of  $\mathcal{P}_{\tilde{g}_{t_0}}$. 
	 	  	
	 	  	We can find a sequence of points $p_n$ on $S_t'$ and $q_n^r$ on $S_{t_0}^r(x_{\max})$ such that $d(p_n,q_n^r)$ tends to zero.  Let $q_n$ be the points on $S_{t_0}(x_{\max})$ that are the images of the $q_n^r$ under normal projection.  Now apply covering transformations $\gamma_n$ to take the $p_n$ back to a fixed compact set containing a fundamental domain for the action of $\pi_1(M)$ on $(\tilde{M},\tilde{g}_{t})$.  Let $p$ and $q$ be subsequential limits of the $\gamma_n \cdot p_n$ and $\gamma_n \cdot q_n$ respectively. Since $S_{t}'$ and $S_{t_0}(x_{\max})$ are minimal surfaces in $\tilde{g}_t$ and $\tilde{g}_{t_0}$ with bounded principal curvatures, we can pass to a subsequence on which $\gamma_n \cdot S_{t}'$ and $\gamma_n \cdot S_{t_0}(x_{\max})$ are smoothly and graphically converging in small balls centered at $p$ and $q$ respectively.  The $r$-neighborhood of the subsequential limit of the $\gamma_n \cdot S_{t_0}(x_{\max})$, whose boundary is strictly mean-convex, will then touch the subsequential limit of the $\gamma_n \cdot S_t'$, which is minimal, on one side at $p$, which is a contradiction.

	 \end{proof}

%	We claim that for every $\epsilon$ there is some $\xi$ such that for any minimal properly embedded $\epsilon$-convex surface $\Sigma$ in any $(\tilde{M}, \tilde{g_\tau})$, signed $r$-neighborhoods of $\Sigma$ are strictly mean-convex for $0< |r|< \xi$. This follows from Equation \ref{riccati}. To see this, first note that we have uniform bounds on all norms of derivatives of the Ricci curvature tensor over all $g_t$.  Also, since $\Sigma$ is minimal, standard elliptic PDE theory implies bounds on all derivatives of the principal curvatures of $\Sigma$ depending only on bounds on the geometry of $g(t)$. This implies first of all a uniform bound on the size of the normal injectivity radius of $\Sigma$-- i.e., a lower bound for a $\delta$ such that the normal exponential map on the normal bundle to $\Sigma$ is injective restricted to $\Sigma \times (-\delta,\delta)$. It also implies bounds on the sizes of the derivatives of the principal curvatures of parallel distance $r'$ surfaces to $\Sigma$ along geodesics normal to $\Sigma$, for  $|r'|$ less than some $r$ that only depends on bounds on the geometry of the $g(t)$ and $\epsilon$. Equation \ref{riccati} then implies that parallel distance $r$ surfaces are mean-convex for $r$ less than some $\xi$ that depends only on $\epsilon$ and the geometry of the $g_t$.  

 Since $S_t'$ is $\epsilon$-subordinate, signed-distance-$r$ surfaces to $S_t'$ are strictly mean-convex for $0< |r|< \xi$, for the $\xi$ given by Lemma \ref{lemconvex}. By the previous lemma  $S_t'$ is contained in the $\xi/4$ $\tilde{g}_{t_0}$-neighborhood of $S_{t_0}$, so the $\tilde{g}_{t_0}$-normal projection from $S_t'$ to $S_{t_0}$ is well-defined.  Since $S_t'$ and $S_{t_0}$ are properly embedded, the normal projection is a proper local diffeomorphism, and consequently surjective. Therefore, because $S_t$ is contained in the $\xi/4$ $\tilde{g}_{t_0}$-neighborhood of $S_{t_0}$, it follows that $S_t$ is contained in the  $\xi/2$ $\tilde{g}_{t_0}$-neighborhood of $S_t'$, and therefore, as long as $\delta$ was taken small enough, the $3\xi/4$  $\tilde{g}_{t}$-neighborhood of $S_t'$.  If $S_t$ were not equal to $S_t'$, then we could take a sequence of points on $S_t$ approaching a supremal mean-convex parallel surface to $S_t'$ and produce a contradiction as in the proof of the previous lemma.  Therefore $S_t$ is equal to $S_t'$, and in particular is $\epsilon$-subordinate. 
	
	In a similar way, we can check that $S_t$ is the unique properly embedded minimal surface at finite Hausdorff distance from $S$.  The argument in Lemma \ref{lim} shows that any such minimal surface must be contained in the $r$-neighborhood of $S_t$, and from there we can show it must be equal to $S_t$ by the reasoning of the previous paragraph.

%	, provided $r$ was taken less than $\frac{\xi}{2}$.

	\subsection{Property 2} \label{fact 2}    
	
	We now check that the surfaces $S_t$ we have constructed satisfy Property 2. Assume for contradiction that they do not. Then there is some sequence $S_n$ of totally geodesic planes that converges to $S$ on compact subsets, while the corresponding sequence of minimal surfaces $S_{n,t}$ in $\tilde{g_t}$ is not smoothly converging to the minimal plane $S_t$ corresponding to $S$.  Since the $S_{n,t}$ are minimal, $C^1$ convergence of the $S_{n,t}$ on compact sets would imply smooth convergence, so we can assume that there is some $\eta>0$ such that the lifts of the $S_{n,t}$ to the unit tangent bundle by their normal vectors all have points in some fixed compact subset of $\tilde{M}$ at a distance of at least $\eta$ from the lift of $S_t$ to the unit tangent bundle.    
	
	The sequence $S_{n,t_0}$ of $\tilde{g}_{t_0}$ minimal surfaces corresponding to the $S_n$ converges smoothly to the minimal surface $S_{t_0}$ corresponding to $S$ by assumption.  %This means that the nearest point projection from $S_{n,t}$ to $S_{n,t_0}$ is a diffeomorphism onto its image.  
	For every compact set $B$ in $(\tilde{M}, \tilde{g_t})$ %there exists $N$ such that for $n>N$ the nearest point projection from $S_{n,t_0}\cap B$ to $S_{t_0}$ is a diffeomorphism onto its image. 
	the intersection $B\cap S_{n,t}$ is therefore, in normal exponential coordinates for the $\xi$-neighborhood of $S_{t_0}$, a graph over $S_{t_0}$ for large enough $n$, since $S_{n,t}$ is a graph over $S_{n,t_0}$.  Because we have uniform bounds on the second fundamental forms of the $S_{n,t}$ by Lemma \ref{schoen}, we can then proceed exactly as in	{\ref{construction}. above to pass to a subsequential limit, which is a properly embedded minimal surface at finite distance from $S_t$ and so must equal $S_t$ by the uniqueness properties of $S_t$ verified in \ref{uniqsection}. But since the points on the $S_{n,t}$ where the normal vectors are at least $\delta$ from any normal vector to $S_t$ are contained in a compact set, we can find an accumulation point of any infinite sequence of them. This contradicts equality of the limit minimal surface with $S_t$.

	\subsection{The $S_t$ Give a Foliation} 
	
	We now check that the $S_t$ give a foliation of $Gr_2(\tilde{M})$ that is invariant under the action of $\pi_1(M)$.  First, the set of $S_t$ is invariant under covering transformations. This is because we've already checked that each $S_t$ is the unique properly embedded minimal disk at finite distance from some element of $ \mathcal{P}_{\tilde{g}_t}$--- that is, some  totally geodesic plane in $\mathbb{H}^3$ considered as a subspace of $(\tilde{M},\tilde{g}_t)$--- and the set $ \mathcal{P}_{\tilde{g}_t}$ is invariant under covering transformations.

	By our inductive hypothesis, there exists a continuous self-homeomorphism $\tilde{\Phi}_{t_0}$ of $Gr_2(\tilde{M})$ sending (lifts of) totally geodesic planes in $\mathbb{H}^3$ to the corresponding (lifts of) minimal disks $S_{t_0}$.  Since nearest-point projection defines a diffeomorphism between the $S_t$ and the corresponding $S_{t_0}$, by composing with $\tilde{\Phi}_{t_0}$ we obtain a self-map $\tilde{\Phi}$ of $Gr_2(\tilde{M})$ diffeomorphically sending (lifts of) totally geodesic planes to (lifts of) $S_{t}$.  That $\tilde{\Phi}$ is continuous follows from the fact that the $S_t$ satisfy Property 2. Note also that $\tilde{\Phi}$ commutes with diffeomorphisms of $Gr_2(\tilde{M})$ induced by covering transformations of $\tilde{M}$.  This follows from the fact that this is true for $\tilde{\Phi}_{t_0}$, and that since the $S_t$ and $S_{t_0}$ are invariant under covering transformations, nearest-point projection from $S_t$ to $S_{t_0}$ commutes with covering transformations.  
	
	This shows that $\tilde{\Phi}$ descends to a continuous self-map of $Gr_2(M)$, which since this map is $O(\xi)$-close to the corresponding self-homeomorphism of $Gr_2(M)$, it must be the case, provided $\delta$ (since $\xi = O(\delta)$) was taken small enough, that the self-map of $Gr_2(M)$ induced by $\tilde{\Phi}$ is homotopic to the homeomorphism induced by $\tilde{\Phi}_{t_0}$.  The map $\tilde{\Phi}$ therefore has mod-2 degree one \iffalse(recall that $Gr_2(M) \cong M \times \mathbb{RP}^2$ is not orientable)\fi and so is surjective.  This means that every point of $Gr_2(\tilde{M})$ is the tangent plane of some $S_t$.  The main step remaining to prove that the $S_t$ give a foliation is to check that each point of  $Gr_2(\tilde{M})$ is the tangent plane of a unique $S_t$, or in other words that $\tilde{\Phi}$ is injective.  It will then follow that $\tilde{\Phi}$ is a homeomorphism because a continuous bijection between compact metric spaces is a homeomorphism.

	The main tool for proving injectivity of $\tilde{\Phi}$ will be the following lemma, whose proof is immediate from the results in \cite[Section 5.3]{coldmini}.  This is one of the key places in the paper where we use that the ambient dimension is three.   
	
	\begin{lem} \label{normalform}
		Let $S_1$ and $S_2$ be properly embedded minimal planes in $(\tilde{M}, \tilde{g}_t)$.  Then the intersection $S_1 \cap S_2$ is an embedded graph. At any point where the two intersect non-transversely, the intersection is locally homeomorphic to a union of $n\geq 2$ straight lines with a common point.   
		
	\end{lem}

		%This follows from the results of \cite{meeksyau}.  

		%\cite{freehassscott}

	Let $S$ and $S'$ be totally geodesic disks in $\mathbb{H}^3$, and let $S_t$ and $S_t'$ be the corresponding minimal disks in $(\tilde{M},\tilde{g}_t\}$.

	 	 Then by Lemma \ref{normalform}, $S_t$ and $S_t'$ intersect in a graph $\Gamma$. To show that they never intersect non-transversely and prove injectivity of $\tilde{\Phi}$, it is enough by the previous lemma to show that $\Gamma$ is either empty or a disjoint union of lines. 
	 	 
	 	 We claim that since $S_t$ is absolutely-minimizing, the set difference $S_t-\Gamma$ cannot have any bounded connected components. For contradiction, assume it had such a connected component $D$, which by taking an innermost such component we can assume is topologically a disk. Then by taking some large circle $C$ in $S_t$ which bounds a disk that contains the boundary of $D$ in $S_t$, we could, by cutting out $D$ and replacing it with the bounded connected component of  $S_t'- \partial D$ (which has the same area as $D$ since all disks in $S_t$ and $S_t'$ minimize area over comparison disks with the same boundary), produce a non-$C^1$ solution to the Plateau problem for $C$ in the ambient space.  This is impossible though, since the area can be decreased by smoothing in neighborhoods of non-$C^1$ points of transverse intersection (\cite[Section 5.3]{coldmini}.)  This shows that the $S_t$ we have constructed satisfy Property 1.   
	 	 
	 	 In the case that $S_t$ and $S_t'$ have disjoint boundaries at infinity, we are done by the last paragraph, since $S_t$ and $S_t'$ do not intersect outside of some compact set, so if the two intersected there would have to be a compact connected component of the complement of the intersection.  
	 	 
	 	 Otherwise, assume $S$ and $S'$ intersect in a line. Assume for contradiction that $S_t$ and $S_t'$ intersect non-transversely at a point $p$, and let $\Gamma_0$ be the connected component of $\Gamma$ containing $p$.  Locally at $p$, $\Gamma_0$ looks like $n>1$ lines meeting at a point. If there are more non-transverse intersections these lines might further branch, but they will never intersect each other at a point besides the initial branch point since that would create a compact connected component of their complement.  Since $S_t$ and $S_t'$ are at finite Hausdorff distance from $S$ and $S'$ respectively, all points on $\Gamma$ are at uniformly bounded distance from $S \cap S'$. It follows that the complement of $\Gamma_0$ in $S_t$ has a connected component $D_0$ all of whose points are at uniformly bounded distance from $S_t'$. Since the proof of Lemma \ref{vertfoliation} only used Properties 1 and 2 which we have already verified, there is a  foliation $\mathscr{F}$ of $\tilde{M}$ containing $S_t'$ as a leaf by applying that lemma to $S_t'$ and $S'$. 
	 	 
	 	 Without loss of generality, assume that $D_0$ has points above $S_t'$, and let $d$ be the supremum of the set of distances from points in $D_0$ above $S_t'$ to $S_t'$.    Since the $\xi$-neighborhood of $S_t'$ has a local mean-convex foliation, if $d$ were less than $\xi/2$, we could get a contradiction by the argument of Lemma \ref{lim}. Otherwise, we could choose another leaf $S_t''$ in the foliation $\mathscr{F}$ above $S_t'$ so that all points on $D_0$ above $S_t''$ were at a distance of less than $\xi/2$ from $S_t''$. Since the $\xi$-neighborhood of $S_t''$ also has a mean-convex foliation, this would lead to a contradiction in the same way.   (A similar argument shows that $S_t$ and $S_t'$ intersect in a single line, although we only need to show they intersect transversely.)

	 	 % so we can take $D_0'$ to be a connected component of $S_t-G_0'$ between two consecutive ``prongs" of $G_0'$ that stay at $o(\epsilon)$ distance from $S_t \cap S_t'$.
	 	 
	 	 %Then $G_0'$ cannot contain any other non-transverse intersection points, or there would be a bounded component of the complement of the intersection points in $S_t$. 
	 	 
	 	 %Therefore $G_0'$ is homeomorphic to $n>1$ lines that meet at a single point corresponding to $p$.   

	  % Now let $S_t'(t)'$ be the strictly mean convex foliation of the unit neighborhooof $S_t'$ given by Lemma .   Without loss of generality assume $D_0'$ lies in the same connected component of $(\mathbb{H}^3,g_\epsilon) - S_t'$ as the $S_t'(t)'$ for $t>0$. Then since $D_0'$ stays a $o(\epsilon)$ distance away from $S_t'$, there is some supremum $T<1$ for the set of all $t$ such that $S_t'(t)'$ intersects $D_0'$. If some point attains the supremum this gives a contradiction, otherwise there is a sequence of points $p_n$ on $S_t'(t_n)' \cap D_0'$ with $t_n \rightarrow T$.  We can derive a contradiction by the same argument as in the uniqueness part of the proof of  by applying isometries $\gamma_n$ that move the $p_n$ back to a fundamental do and taking a convergent subsequence of the translated copies of $S_t$ intersected with some ball.  

	 	 The only case left to check is if $S_t$ and $S_t'$ intersect at a single point at infinity.  The proof here is like the last case.  If $S_t$ and $S_t'$ intersect non-transversely at some point, then we can similarly deduce the existence of some unbounded connected component of the intersection $\Gamma$ all of whose points are at bounded distance from $S_t'$.  We can then produce a contradiction as above by taking a mean-convex foliation of a neighborhood of $S_t'$, or else some other $S_t''$ above or below $S_t'$.

\subsection{Local Product Charts}

	 	 A smooth local product chart for our foliation at any $p$ in $\tilde{M}$ and any tangent plane $P$ to $p$ in $Gr_2(\tilde{M})$ can be constructed as follows.  Let $S_t$ be the surface which has $P$ as a tangent plane. The transversal to our chart will be homeomorphic to the product of a small neighborhood $U$ of $P$ in the Grassmannian of the tangent space $Gr_2(T_p \tilde{M})$ with a small geodesic segment $\gamma$ in $\tilde{M}$ containing $p$ and normal to $P$ at $p$.  We diffeomorphically identify this product with a subspace $T$ of $Gr_2(\tilde{M})$ by parallel transporting $U$ along $\gamma$.  
	 	 
	 	 Take a small metric disk $V$ centered at the origin in the tangent space to $S_t$ at $p$.  We construct a map from $V \times T$ to  a small neighborhood of $P$ in $Gr_2(\tilde{M})$ as follows.

	 	 Let $(v,(p',P'))\in V \times T$ be given.  We can identify $P'$, by parallel transport along $\gamma$, with a linear subspace $P'_p$ of $T_p(\tilde{M})$.  Provided $U$ and $\gamma$ were chosen sufficiently small, the normal projection $v'$ to $P'_p$ of $v$ will have norm greater than $\frac{1}{2}|v|$.

	    Parallel transport of tangent vectors gives a natural identification between $P'_p$ and $P'$ viewed as subspaces of $T_{p} \tilde{M}$ and $T_{p'} \tilde{M}$ respectively.  Take the vector in $P'$ corresponding to $v'$--- call it $v''$--- and consider $v''$ inside the tangent space at $p'$ to the surface $S_t'$ that has $P'$ as a tangent plane.   We map $v''$ to its image under the exponential map of $S_t'$ at $p'$ in the metric on $S_t'$ induced by $\tilde{g}_t$ and define the tangent plane to $S_t'$ of this point to be the image of  $(v,(p',P'))$ in $Gr_2(\tilde{M})$ under our map.  
	    
	    %By scaling this vector so that its norm in $T_{p'}S_t'$ is $|v|$, we obtain a new vector $v''$.
	    
	    Knowing that the surfaces $S_t$ vary smoothly in their tangent planes, we will know that their exponential maps vary smoothly, and smoothness of the coordinate map we have defined will follow.  Suppose that a sequence $S_{t,n}$ of minimal disks we have constructed has tangent planes $P_n$ converging to the tangent plane $P$ to $S_t$ at a point. Then we need to show that $S_{t,n}$ is smoothly converging to $S_t$ on compact subsets. That the convergence is $C^1$ follows from the fact that $\tilde{\Phi}$ is a homeomorphism, and elliptic PDE theory implies that the convergence is smooth.  
	    
	    Since the differential of the coordinate map at $(0,(p,P))$ is non-singular, we can apply the inverse function theorem to restrict to a possibly smaller neighborhood of $(0,(p,P))$ in  $V \times T$ on which it is a diffeomorphism onto its image.  This shows that every point in $Gr_2(\tilde{M})$ is contained in a smooth product chart for the foliation.  The proof of Theorem \ref{main} is now complete.  
	    
	    % This follows from the fact that the map $\Phi$ we have defined that maps totally geodesic planes to the $S_t$ is a homeomorphism. The fact that the corresponding sequence $S_n$ of totally geodesic planes has tangent planes converging to a tangent plane to $S$ implies that the $S_n$ are converging to $S$ on compact subsets.  Since $\Phi$ is a homeomorphism, this implies that the normal vectors are of the $S_{t,n}$ are converging uniformly to those of $S_t$ on compact sets (in other words $S_{t,n}$ is converging $C^1$-uniformly to $S_{n}$ on compact sets,) and from there elliptic estimates imply that the convergence is smooth.     	

	\subsection{Proof of Theorem \ref{mainintro}} 
	
	We now explain how to modify the proof of Theorem \ref{main} to give a proof of Theorem \ref{mainintro}.  Let $g_t$ be a smooth family of metrics as in the statement of Theorem \ref{mainintro}.  Then if $g_{t_0}$ is in $\Omega_\epsilon$, we claim that there is some $\delta$ depending only $\epsilon$ and bounds on the geometry of the $g_t$ such that $g_t$ is in $\Omega_{\epsilon/2}$ for $|t-t_0|< \delta$. The $\tilde{g}_t$-minimal surfaces $S_t$ in $\tilde{M}$ can be constructed exactly as above. Each of the $S_t$ can be made as $C^1$-close to the corresponding $S_{t_0}$ as desired, uniformly in $S_t$, by making $\delta$ small.  Elliptic PDE theory tells us that $C^1$-close implies $C^2$-close, so if we chose $\delta$ small enough, the $S_t$ will be $\epsilon/2$ subordinate.  The metric $g_t$ will then be contained in $\Omega_{\epsilon/2}$, and the verification that the $S_t$ give a foliation and the construction of the map $\tilde{\Phi}$ can proceed as above. % The rest of the construction of the foliation can proceed exactly as above, and so 
	
	The construction of the foliations thus continues to work unless there is some time $T$ such that for every $\epsilon>0$ and  sequence of times $t_n \nearrow T$  there is some $S_{t_n}$ in $(\tilde{M},\tilde{g}_{t_n})$ that fails to be $\epsilon$-subordinate for $n$ large enough. This proves Theorem \ref{mainintro}.

\section{Applications} \label{applications}

\subsection{Density}   
In this section, we prove some density results for the stable immersed minimal surfaces in $M$ corresponding to the surface subgroups constructed by Kahn and Markovic. 

%That is, assume that $g$ can be joined to the hyperbolic metric $g_{hyp}$ on $M$ through a smooth path of metrics all contained in some fixed $\Omega_\epsilon$. 
Let $g$ be a negatively curved metric on $M$ to which Theorem \ref{mainintro} or Theorem \ref{main} applies to produce a foliation. Let $\mathcal{F}_g$ be the foliation of $Gr_2((M,g))$ by (lifts of) $g$-minimal immersed disks, and let 
\[
\Phi: Gr_2((M,g_{hyp} )) \rightarrow Gr_2((M,g))
\]

\noindent be the conjugating homeomorphism that sends leaves of $\mathcal{F}_{g_{hyp}}$ to leaves of $\mathcal{F}_g$. Every leaf of $\mathcal{F}_{g_{hyp}}$ is either dense or properly immersed, so since $\Phi$ is a homeomorphism the leaves of $\mathcal{F}_g$ satisfy the same dichotomy.

%Kahn-Markovic showed that for every circle $C$ at infinity in $\partial_\infty{\mathbb{H}^3}$ there is a sequence of surface subgroups $\Gamma_n$ of $\pi_1(M)$ so that the Hausdorff distance between $C$ and the limit sets $C_n$ of the surface subgroups is tending to zero in $\partial_\infty{\mathbb{H}^3} \cong S^2$ with the round metric.  The $C_n$ are the images of round circles under $K_n$-quasiconformal homeomorphisms of $S^2$, with $K_n \rightarrow 1$.  By results of [Meeks-Yau] and [Sacks-Uhlenbeck], there exists a sequence  $\Sigma_n$ (resp. $\Sigma_n'$) of stable properly immersed $g_{hyp}$-minimal (resp. $g$-minimal) surfaces in $M$ whose fundamental groups injectively include in $\pi_1(M)$ to a subgroup conjugate to $\Gamma_n$.  The main result of [seppi] (see also the appendix to this paper (?)) implies that, for $K_n$ sufficiently close to 1, each $\Sigma_n$ is the unique stable properly immersed minimal surfaces the inclusion of whose fundamental group is conjugate to $\Gamma_n$. The next lemma will be used to show that for large enough $n$ the $\Sigma_n'$ also have this uniqueness property. 

%Let $\tilde{\Sigma_n}'$ be a minimal disk in the universal cover $(\tilde{M},\tilde{g})$ that projects down to $\Sigma_n'$. 

%Let $\Gamma$ be a surface subgroup of $\pi_1(M)$, and let $\Sigma$ and $\Sigma'$ be stable properly immersed respectively $g_{hyp}$-minimal and  $g$-minimal surfaces whose fundamental group includes as a subgroup of $\pi_1(M)$ conjugate to $\Gamma$. 

Fix a compact set $K_0 \subset \tilde{M}$ that contains a small neighborhood of a connected polyhedral fundamental domain for the action of $\pi_1(M)$ on $\tilde{M}\cong \mathbb{H}^3$. 
% The next lemma essentially says that if the limit set of $\Gamma$ is close to a circle, then minimal surfaces corresponding to $\Gamma$ are locally well-approximated by leaves of $\mathcal{F}_g$.    
Let $\Sigma_n$ (resp. $\Sigma_n'$)  be a sequence of stable immersed $g_{hyp}$-minimal (resp. $g$-minimal) surfaces with lifts $\tilde{\Sigma}_n$ (resp. $\tilde{\Sigma}_n'$) to $\tilde{M}$.   %Suppose that the fundamental groups of $\Sigma_n$ and $\Sigma_n'$ include as subgroups of $\pi_1(M)$ conjugate to the surface subgroup $\Gamma_n$ of $\pi_1(M)$. 
Suppose the lifts $\tilde{\Sigma}_n$ and $\tilde{\Sigma}_n'$ were chosen so that all of the $\tilde{\Sigma}_n$ intersect $K_0$, and that $\tilde{\Sigma}_n'$ and $\tilde{\Sigma}_n$ are at finite Hausdorff distance from each other in $\tilde{M}$ in either (or equivalently both) of $\tilde{g}$ or $\tilde{g}_{hyp}$.  %The following lemma is similar to Proposition 4.1 in \cite{cmn}.  

\begin{lem} \label{traintracks} 
 %Let $\tilde{\Sigma}$ be a disk in $(\tilde{M}, \tilde{g}_{hyp})\cong \mathbb{H}^3$ that projects down to $\Sigma$ and intersects $K_0$.  
 
 Fix a circle $C$ in $\partial_\infty\mathbb{H}^3 \cong S^2$, and suppose that the limit sets of the $\pi_1(\tilde{\Sigma}_n)$ are Hausdorff converging to $C$ in $\partial_\infty \mathbb{H}^3 $. Let $L'$ be the minimal disk in $(\tilde{M},\tilde{g})$ whose lift to $Gr_2(\tilde{M})$ is the image under $\tilde{\Phi}$ of the lift of the totally geodesic plane $L$ with limit set $C$. Then the $\tilde{\Sigma}_n'$ converge smoothly to $L'$ uniformly on compact sets.      
 
 %Denote by $\tilde{\Sigma}'$ the unique disk in $\tilde{M}$ that projects to $\Sigma'$ and is at finite Hausdorff distance from $\tilde{\Sigma}$. 

%Then for each point $p \in L'$ at $\tilde{g}$-distance of less than $R$ from $K_0$, there is some point $p' \in \tilde{\Sigma}'$ such that the unit normal vector to $\tilde{\Sigma_n}'$ at $p'$ is at a distance in the unit tangent bundle of less than $\eta$ from the unit normal vector to some point of $\tilde{L}$. 

\end{lem}

%\begin{lem} 
%For every $R>0$ and for every $\eta>0$ there exists $N$ such that if $n>N$ and $p \in \tilde{\Sigma_n}'$ there exists a leaf of $\mathcal{F}_g$ with a lift $\tilde{L}$ to $\tilde{M}$ so that the following holds.  Let $p' \in \tilde{\Sigma_n}'$ be a point in the $\tilde{g}$-ball in $\tilde{M}$ of radius $R$.  Then the unit normal vector to $\tilde{\Sigma_n}'$ at $p'$ is at a distance in the unit tangent bundle of less than $\eta$ from the unit normal vector to some point of $\tilde{L}$.  
%\end{lem}    

\begin{proof} 

%Suppose there are $R$ and $\eta$ for which the statement of the lemma fails. 

 %Assume that the limit sets of $\tilde{\Sigma_n}$ are $1+\frac{1}{n}$ quasi-circles at a Hausdorff distance of less than $\frac{1}{n}$ from a circle $C_n$ in $\partial_\infty{\mathbb{H}^3}$.  Let $L_n'$ be the $\tilde{g}$-minimal disk in $\tilde{M}$ whose lift to $Gr_2(\tilde{M})$ is the image under $\Phi$ of the lift of the totally geodesic plane $L_n$ corresponding to $C_n$.  

%Assume that there are points $p_n \in \tilde{\Sigma_n}'$ contained in some fixed compact set 

%whose unit normal vectors are at least a distance of $\eta$ from any unit normal vector to $L_n'$. 

%Consider a small  $C$ in $\partial_\infty  \mathbb{H}^3$ is foliated by parallel circles whose corresponding totally geodesic planes foliate a region containing $\tilde{\Sigma}_n$. Since at least one such totally geodesic plane intersects $K_0$, and the limit set of a totally geodesic plane intersecting $K_0$ has radius uniformly bounded below in the round metric on $S^2$, there is a uniform lower bound on the radii of the $C_n$.  

%The set of $C_n$ is therefore precompact, so we can, by passing to a subsequence, assume that the $C_n$ are converging to some circle $C$ that bounds a geodesic plane $L$. Let $L'$ be the corresponding minimal disk in $(\tilde{M},\tilde{g})$. We claim that the $\tilde{\Sigma_n}'$ are converging to $L'$ uniformly on compact subsets. 

Let $L'(t)$ be the foliation of $\tilde{M}$ given by Lemma \ref{vertfoliation} with $\tilde{g}$-minimal leaves and $L'(0)=L'$. Let $L(t)$ be the corresponding foliation of $\mathbb{H}^3$ by geodesic planes with $L(0)=L$, so that $\tilde{\Phi}$ sends lifts of the $L(t)$ to lifts of the $L'(t)$. For every $\alpha>0$ and large enough $n$, $\tilde{\Sigma}_n$, which is contained in the convex hull of its limit set, will be contained between $L(\alpha)$ and $L(-\alpha)$.  We claim that  $\tilde{\Sigma}_n'$ is contained between $L'(\alpha)$ and $L'(-\alpha)$. 

Recall that the $\xi$-neighborhood of every $L'(t)$ has a foliation by mean-convex parallel surfaces, where $\xi$ depends on $\epsilon$ and $g$. Now, if $\tilde{\Sigma}_n'$ were not contained between $L'(\alpha)$ and $L'(-\alpha)$, then using the mean-convex parallel surfaces of the $L'(t)$ and the fact that $\tilde{\Sigma}_n'$ and the $L'(t)$ are at uniformly bounded Hausdorff distance from respectively $\tilde{\Sigma}_n$ and the corresponding  $L(t)$, one could produce a contradiction by arguments similar to those of the last section. One can also show by reasoning similar to subsection \ref{fact 2}. of the last section that $\tilde{\Sigma}_n'$ is  $C^1$-converging and thus smoothly converging to $L'=L'(0)$ on compact subsets. %, since the same is true for the $L'(\alpha)$ as $\alpha \rightarrow 0$.  

%Since the $C_n$ are converging to $C$, the $L_n'$ are converging to $L'$. Since all the $p_n$ are contained in a fixed compact set-- the closed $R$-neighborhood of $K_0$-- taking a convergent subsequence gives a contradiction. Therefore for $n$ sufficiently large, the unit normal vector to every point on $\tilde{\Sigma_n}'$ at a distance of less than $R$ from $K_0$ is $\eta$-close to the unit normal vector to some point of $L'$.  Since $\tilde{\Sigma}_n'$ is smoothly converging to $L'$ on compact sets, for $n$ Jordan curves on $\gamma$ on $\tilde{\Sigma}_n'$ very far from $K_0$ will project to Jordan curves on $L'$ very far away from $K_0$.

\end{proof} 

Kahn and Markovic showed that for every circle $C$ at infinity in $\partial_\infty{\mathbb{H}^3}$ there is a sequence of surface subgroups $\Gamma_n$ of $\pi_1(M)$ whose limit sets $C_n$ are Hausdorff converging to $C$ \cite{km}.  The $C_n$ are the images of round circles under $K_n$-quasiconformal homeomorphisms of $S^2$--- or $K_n$-quasicircles--- with $K_n$ tending to 1.  By \cite{sacksuhl} or \cite{schoenyauincompressible}, there exists a sequence  $\Sigma_n$ (resp. $\Sigma_n'$) of stable properly immersed $g_{hyp}$-minimal (resp. $g$-minimal) surfaces in $M$ whose fundamental groups injectively include in $\pi_1(M)$ to subgroups conjugate to $\Gamma_n$.

\begin{thm} 
Let $C$ be a circle in $\partial_\infty \mathbb{H}^3$ bounded by a geodesic plane $L$ in $\mathbb{H}^3$ that does not project to a properly immersed surface in $M$.  Then for any sequence $\Sigma_n$ of stable properly immersed minimal surfaces with lifts $\tilde{\Sigma}_n$ to $\mathbb{H}^3$ whose boundaries at infinity are $K_n$-quasicircles with $K_n$ tending to 1 and Hausdorff converging to $C$, the following is true.  Let $U$ be any open set in $Gr_2(M)$.  Then there exists $N$ so that for $n>N$, the intersection of the lift of $\Sigma_n'$ to $Gr_2(M)$ with $U$ is nonempty.   

\end{thm}

\begin{proof}
By \cite{ratner} or \cite{shah}, the lift to $Gr_2(M)$ of the covering projection of $L$ to $M$ is dense.  Let (the lift of) $L'$ be the image under $\tilde{\Phi}$ of the (lift of) $L$.  Then as observed earlier in the section, $L'$ is also dense in $Gr_2(M)$. It is enough to prove the theorem for $U$ a small ball in $Gr_2(M)$.  For any such $U$, let $\tilde{U}$ be a lift of $U$ to $Gr_2(\tilde{M})$.  Then we can find $\gamma \in \pi_1(M)$ so that the image  $\gamma \cdot L'$ of (the lift to $Gr_2(M)$ of) $L'$ under the covering transformation corresponding to $\gamma$ intersects $\tilde{U}$.  

%Going back to $\mathbb{H}^3$, note that the limit sets of $\gamma \cdot \tilde{\Sigma}_n$ are tending to $\gamma \cdot C$ in Hausdorff distance.  Since 
By Lemma \ref{traintracks} the $\tilde{\Sigma}_n'$ are smoothly converging to $L'$ on compact sets, so for all $n$ sufficiently large $\gamma \cdot \tilde{\Sigma}_n'$ will intersect $\tilde{U}$. Therefore for all sufficiently large $n$, %,  $\gamma \cdot \tilde{\Sigma}_n'$ will intersect $\tilde{U}$, and 
$\Sigma_n'$ will intersect $U$.

% Lemma \ref{traintracks} then implies that the $\gamma \cdot \tilde{\Sigma}_n'$ will be $C^1$-converging to $\gamma \cdot L'$ uniformly on compact sets.     
	
\end{proof}

\subsection{Uniqueness}  

We now prove uniqueness for properly immersed minimal surfaces whose fundamental groups injectively include to the conjugacy class of a given surface subgroup of $\pi_1(M)$, under the assumption that the limit set of the surface subgroup is close to a circle.   

Let $\Sigma$ be a stable properly immersed minimal surface in $(M,g_{hyp})$ whose fundamental group injectively includes in $\pi_1(M)$ as a subgroup conjugate to a surface subgroup $\Gamma$ of $\pi_1(M)$.  Then if the limit set of $\Gamma$ is a $K$-quasicircle for $K$ sufficiently close to 1, the main result of \cite{seppi} implies that $\Sigma$ will be the unique such surface.  The next theorem is an analogous uniqueness result for $(M,g)$. The proof occupies the remainder of the section.   %For its statement, all notation is as in the previous lemma.    

\begin{thm} \label{uniqeness} 
Fix $(M,g)$ to which Theorem \ref{mainintro} or \ref{main} applies to produce a foliation. Then there exists $\eta>0$ such that the following is true.  Suppose the limit set $\partial_\infty \tilde{\Sigma}$ of $\tilde{\Sigma}$ is a $K$-quasicircle for $K<1+\eta$.  Then there is a unique closed $g$-minimal surface in $M$ whose fundamental group injectively includes in $\pi_1(M)$ as a subgroup conjugate to $\Gamma$.   

%, and that $\partial_\infty \tilde{\Sigma}$ is at a Hausdorff distance of less than $\delta$ from a circle in $S^2$ that bounds a totally geodesic plane that intersects $K_0$.   
\end{thm}

\subsubsection{Construction of the Maps $\tilde{f}_{\Sigma}$} 
  	
In \cite{seppi}, Seppi produces, for every point $p$ on a minimal disc $D$ with limit set a $1+\eta$ quasicircle, planes $L_1(p)$ and $L_2(p)$ respectively above and below $D$ such that the quantity 
\begin{equation} \label{close} 
\max (d(p,L_1(p)), d(p,L_2(p)))    
\end{equation} 

\noindent tends to zero uniformly in $p$ as $\eta$ tends to zero. This builds on work of \cite{krasnovschlenker} and \cite{epsenv}.  Taking $D=\tilde{\Sigma}$, we can choose $L_1(p )$ and $L_2(p)$ such that the arc-length-parametrized geodesic $\gamma_p$ normal to $\tilde{\Sigma}$ at $p$ perpendicularly intersects $L_1(p)$ and $L_2(p)$ at $\gamma_p(\delta)$ and $\gamma_p(-\delta)$ respectively, where $\delta$ is independent of $p$ and can be made as small as desired by making $\eta$ small. In addition, %the main result of \cite{seppi}--- that the principal curvatures of $\tilde{\Sigma}$ tend to zero as $\eta$ tends to zero--- and 
Lemma \ref{normalvector}  implies that $L_1(p)$, $L_2(p)$, and $\delta$ can be chosen so that the following is true: the distance from the tangent plane to $\tilde{\Sigma}$ at $p$ from the tangent planes to $L_1(p)$ and $L_2(p)$ at $\gamma_p(\delta)$ and $\gamma_p(-\delta)$ tends to zero as $\eta$ tends to zero, uniformly in $p$.        

%Note that the map $\tilde{\Phi}: Gr_2((\tilde{M}, \tilde{g}_{hyp})) \rightarrow     Gr_2((\tilde{M},\tilde{g}))$ is uniformly continuous as the lift of a map  $\Phi: Gr_2((M, g_{hyp})) \rightarrow     Gr_2((M,g))$.  %This implies that for every $\epsilon$ there exists $\delta$ such that the following holds: If $L_1'$ and $L_2'$ correspond to $L_1$ and $L_2$ under $\Phi$, and if $P_1$ and $P_2$ are tangent planes to $L_1$ and $L_2$ such that $d(P_1,P_2)< \delta$, then $d(\Phi(P_1),\Phi(P_2))< \epsilon$. 

 Let $\Sigma'$ be a stable $\pi_1$-injective $g$-minimal surface the inclusion of whose fundamental group in $\pi_1(M)$ is conjugate to that of $\pi_1(\Sigma)$.  Fix some $p\in \tilde{\Sigma}$.  Let $P_1$ and $P_2$ be the tangent planes to $L_1(p)$ and $L_{2}(p)$ at $\gamma_{p}(\delta)$ and $\gamma_{p}(-\delta)$. Then if $\tilde{\Sigma}$ is contained between $L_1(p)$ and $L_2(p)$, we assume that the lift $\tilde{\Sigma}'$ is chosen so that it is contained between the $\tilde{g}$-minimal planes $L_1'(p)$ and $L_2'(p)$ which correspond under $\tilde{\Phi}$ to $L_1(p)$ and $L_2(p)$. We take $p'$ to be a point on the intersection of $\tilde{\Sigma}'$ with the geodesic segment $\gamma_{p}'$ joining the projections to $\tilde{M}$ of $\tilde{\Phi}(P_1)$ and $\tilde{\Phi}(P_2)$.   
 
 We claim that $p'$ is the unique such point of intersection provided $\eta$ and $\delta$ were chosen small enough.  Assume there were some other point of intersection $p''$, and let $\phi'$ be the unique geodesic segment on $\tilde{\Sigma}'$ joining $p'$ to $p''$.

 On the one hand, given any $R>0$ and small $\epsilon>0$, we claim that we can ensure that the intrinsic ball $B_R(p')$ of radius $R$ in $\tilde{\Sigma}'$ centered at $p'$ will be at a $C^1$ distance of less than $\epsilon$ from the ball of that radius in either $L_1'(p)$ or $L_2'(p)$ centered at the intersection with $\gamma_{p'}$, provided that $\eta$ and $\delta$ were chosen sufficiently small. 
 
 %\begin{itemize} 
 %\item $\eta$ was chosen sufficiently small (i.e., the limit set of $\pi_1(\Sigma)$ is a $1+\eta$-quasicircle for $\eta$ small.)  	
 %\item The number $\delta>0$ was chosen sufficiently small.  
 %\end{itemize} 

This follows from the same argument by contradiction as the proof of Lemma \ref{traintracks}.  As a consequence of this, provided $\epsilon>0$ and $\delta>0$ were taken small enough that $B_R(p')$ is sufficiently $C^1$ close to $L_1'(p)$ and $L_2'(p)$, the length of $\phi'$ must be at least $R$.    

On the other hand, the fact that $\tilde{\Phi}$ is a uniformly bi-Lipschitz homeomorphism from $(\tilde{M},\tilde{g}_{hyp})\rightarrow (\tilde{M},\tilde{g})$ implies that for $R$ large enough (and independent of $\tilde{\Sigma}'$), all points in $\tilde{\Sigma}'$ at intrinsic $\tilde{\Sigma}'$-distance of at least $R$ from $p'$ are at $(\tilde{M},\tilde{g})$-distance from $p'$ of at least 1. This is a contradiction.  

Therefore  $\gamma_{p}'$ and $\tilde{\Sigma}'$ intersect at a single point provided $\eta$ and $\delta$ were chosen sufficiently small.  The map $\tilde{f}_{\Sigma}$ that sends a point $p$ to the point $p'$ defined as above is therefore a well-defined map between $\tilde{\Sigma}$ and $\tilde{\Sigma}'$. Because $\tilde{f}_\Sigma$ can locally be made arbitrarily $C^1$-close to $\tilde{\Phi}$ by making $\eta$ and $\delta$ small, we can assume that $\tilde{f}_{\Sigma}$ is a local diffeomorphism.   

Because $\tilde{\Phi}$ commutes with the action of $\pi_1(M)$ on $\tilde{M}$ by isometries of $\tilde{g}_{hyp}$ and $\tilde{g}$, $\tilde{f}_{\Sigma}$ commutes with the action of $\pi_1(\Sigma)$ and descends to a map $f_{\Sigma}: \Sigma \rightarrow \Sigma'$.  Since $f_\Sigma$ is a local diffeomorphism and $\Sigma$ and $\Sigma'$ are homeomorphic, it must be a diffeomorphism, and in particular $\tilde{f}_\Sigma$ is surjective.  
\subsubsection{Proof of Theorem \ref{uniqeness}} 
Suppose that the projections of the leaves of $\mathcal{F}_g$ to $M$ are $\epsilon$-subordinate.  Then $\tilde{\Sigma}'$ is $\frac{\epsilon}{2}$-subordinate provided $\eta$ and $\delta$ were taken sufficiently small.  This is because, as noted above, $\tilde{\Sigma'}$ is locally $C^1$-converging and thus smoothly converging to (projections of) leaves of $\mathcal{F}_g$.  
 
 Now let $\eta$ and $\delta$ be small enough to satisfy all of the restrictions above, as well as one further restriction we will make below in a moment.  Assume for contradiction that $\Sigma'$ and $\Sigma''$ are distinct $g$-minimal surfaces the inclusions of whose fundamental groups in $\pi_1(M)$ are both injective and conjugate to the same surface subgroup $\Gamma$ whose limit set in $\partial_{\infty} \mathbb{H}^3$ is a $K< 1+ \eta$ quasicircle.  

Let $\tilde{\Sigma}'$ and  $\tilde{\Sigma}''$ be lifts to $\tilde{M}$ at finite distance from $\tilde{\Sigma}$ considered inside $(\tilde{M},\tilde{g})$, and note that since these are $\frac{\epsilon}{2}$-subordinate, there exists a uniform $\xi$ such that the $\xi$-neighborhood of each has a mean-convex foliation by parallel surfaces. This implies that $\tilde{\Sigma}'$ and $\tilde{\Sigma}''$ are at a Hausdorff distance of at least $\xi$ from each other.  

For each of $\tilde{\Sigma}'$ and  $\tilde{\Sigma}''$ we have a map $\tilde{f}_{\Sigma}$ defined as above.  Since the definition of the two maps is the same up until taking the intersection with a geodesic segment in $(\tilde{M},\tilde{g})$ of length that tends to zero as $\eta$ and $\delta$ tend to zero, the two images of each point on $\tilde{\Sigma}$ under the two maps will be will be at a distance that tends to zero as $\eta$ and $\delta$ tend to zero. Taking $\eta$ and $\delta$ small enough to make this distance less than $\xi/2$, the fact that both maps $\tilde{f}_{\Sigma}$ are surjective gives a contradiction.  This completes the proof of Theorem \ref{uniqeness}.

%\subsection{} \label{equivariance} To conclude this section, we show that $\tilde{f}_\Sigma$ commutes with the action of $\pi_1(\Sigma)$ and thus descends to a map $f_{\Sigma}: \Sigma \rightarrow \Sigma'$.  The map $f_{\Sigma}$ will be important in the next section. 

%First, as a proper local homeomorphism, the map $\tilde{f}_\Sigma$ is a homeomorphism.  The fact that $\tilde{\Phi}$ commutes with the action of $\pi_1(\Sigma)$ and the definition of $\tilde{f}_{\Sigma}$   

%Recall the above procedure for defining $p_0'$ in terms of $p_0$ by intersecting the geodesic segment $\gamma_{p_0}'$ with $\tilde{\Sigma}'$.     

%To start out, we recall that the map $\tilde{\Phi}$ maps points in tangent planes in $Gr_2((\tilde{M},\tilde{g}_{hyp}))\cong Gr_2(\mathbb{H}^3$  to tangent planes in $Gr_2((\tilde{M},\tilde{g}))$ whose base-points (the points whose tangent spaces the planes lie in) in $\tilde{M}$ are at distance from each other (measured in either $\tilde{g}$ or $\tilde{g}_{hyp}$) bounded by a uniform constant $C$.  

%We also have that 

\section{Quantitative Density} %if $M$ has no totally geodesic surfaces} 
\label{stronger}

%The standing assumption in this section will be that $M$ is a closed hyperbolic 3-manifold that contains no properly immersed totally geodesic surfaces in its constant curvature metric.  For these $M$, we will prove some quantitative versions of the density statements of the previous section.  

%There are many closed arithmetic hyperbolic 3-manifolds with this property cite{nototallygeodesic}. We remark that it is was also recently proven that any non-arithmetic hyperbolic 3-manifold has only finitely many properly immersed totally geodesic surfaces %cite{momarg}, cite{bfms}.

\subsection{Constant Curvature }  \hspace{2mm} 

 \noindent We begin with the constant curvature case.  We assume throughout this subsection that  $M$ contains no properly immersed totally geodesic surfaces in its constant curvature metric.

\begin{defn} 
For a tangent plane $P$ in $Gr_2(M)$ based at $p \in M$, we define $C_{P,r} \subset Gr_2(M)$ as follows.  Lift $(p,P)$ to a point $(\tilde{p},\tilde{P})$ in $Gr_2(\tilde{M}) \cong Gr_2(\mathbb{H}^3)$, and let $\Pi \subset \tilde{M}$ be the geodesic plane tangent to $\tilde{p}$ at $\tilde{P}$.  Now take the circle $\tilde{C}_{P,r} $ in $\Pi$ of radius $r$ centered at $\tilde{p}$, lift it to $Gr_2(\tilde{M})$ by planes tangent to $\Pi$, and let $\tilde{D}_{P,r}$ be the totally geodesic disk that $\tilde{C}_{P,r}$ bounds. We define $C_{P,r}$ and $D_{P,r}$ to be the projections of $\tilde{C}_{P,r}$ and $\tilde{D}_{P,r}$ to $Gr_2(M)$, and we define $\mu_{P,r}$ to be the probability measure that corresponds to averaging over $C_{P,r}$ parametrized by arc-length in the metric induced by $Gr_2(M)$.

\end{defn}

\begin{prop} \label{quantequi} Let 
	\[
	f: Gr_2(M) \rightarrow \mathbb{R}
	\]
	
	\noindent be a continuous function. Then for every $\epsilon>0$ there exists $R$ such that for all $P \in Gr_2(M)$ and all $r>R$, 
	\begin{equation}\label{equi}  
	\left| \int_{Gr_2(M)} f d \mu_{P,r} - avg(f) \right|< \epsilon,  
	\end{equation} 
	where $avg(f)$ is the average of $f$ over $Gr_2(M)$ in its volume form for the metric induced by the hyperbolic metric on $M$.  
	
	\end{prop} 
	
%	\noindent The author thanks Alex Eskin who explained to him an argument similar to the one in the proof below (any mistakes here are the author's.)  
	
	\begin{proof} 
		We are going to prove this by applying Ratner's measure classification theorem (\cite[Theorem 1.11]{pisanotes}.)  Fix an orientation for $C_{P,r}$ and let $\hat{C}_{P,r}$ be its natural lift to the frame bundle $F\cong PSL(2,\mathbb{C})/\pi_1(M)$ of $M$ by, for a point $p \in C_{P,r}$, taking the first vector in the frame to be the outward unit normal vector to the projection of $C_{P,r}$ to $M$ tangent to (the projection of) $\Pi$, and the second to be tangent to $C_{P,r}$ in the direction determined by its orientation. The third vector is then determined by the orientation of $M$. Let $\hat{\mu}_{P,r}$ be the probability measure on $F$ given by averaging over $\hat{C}_{P,r}$.    
		
		Assume that the statement is false, and that for some $f$ and $\epsilon$ there existed a sequence $\mu_{P_n,r_n}$ such that  (\ref{equi}) fails for all $n$.  We can pull back $f$ to a function $\hat{f}$ on $F$ so that

			\begin{equation}\label{equiF} 
				\left| \int_{F} \hat{f} d \hat{\mu}_{P_n,r_n} - avg(f) \right|\geq \epsilon  
			\end{equation}

		  \noindent for all $n$. Since $F$ is compact, we can take a weak-$*$ limit of the $\hat{\mu}_{P_n,r_n}$ to obtain a probability measure $\hat{\mu}$ for which the $\hat{\mu}$-average and the Haar-measure-average of $\hat{f}$ differ by at least $\epsilon$.  
		
		Let $U$ be the projection to $PSL(2,\mathbb{R})$ of the unipotent subgroup of $SL(2,\mathbb{R})$ 			
		\[
		 \left\{\begin{matrix} \begin{pmatrix} 1 & t \\ 0 & 1 \end{pmatrix} \end{matrix} : t \in \mathbb{R} \right\}.  
		\]
		
		\noindent \textbf{Claim: $\hat{\mu}$ is $U$-invariant} 
		
		 %This follows from the fact that larger and larger metric circles in the hyperbolic plane are locally converging to horocycles, which are preserved by $U$.  
		 Fix an element $u \in U$.  Then for every $\delta>0$, there exists $N$ such that for $n>N$, there are arc-length parametrizations $\phi_1$ and $\phi_2$ of respectively $u \cdot \hat{C}_{P_n,r_n}$ and $\hat{C}_{P_n,r_n}$ %by a circle of radius $r_n$ 
		 such that %the distance in $F$ satisfies 
		 \begin{equation} \label{casi} 
		 d(\phi_1(s), \phi_2(s))< \delta
		 \end{equation} 
		for all $s$, where $d$ is the distance in $F$.  This follows from the fact that metric circles of large radius in $\mathbb{H}^2$ can be approximated in large neighborhoods of each point by horocyles, which are preserved by $U$.  %in the upper-half plane model are Euclidean circles of larger and larger Euclidean radius. Any of these circles will be close to a horizontal line near its uppermost point, and horizontal lines in the upper half-plane model of $\mathbb{H}^2$ are horocyles on which $U$ acts by Euclidean translation.  The translation length of any fixed element of $U$ restricted to horizontal-line horocycles in the metric induced from the hyperbolic metric tends to zero as the  $y$-coordinate of the horizontal line tends to infinity. One obtains the almost-invariance under $U$ for all points on the circle, not just the uppermost point, by conjugating by hyperbolic rotations that fix the center of the circles.  
		
		 The inequality (\ref{casi}) implies that for any continuous $g:F \rightarrow \mathbb{R}$, 
		 \[
		\left |\int_F g d (u_* \hat{\mu}_{P_n,r_n}) - \int_F g d \hat{\mu}_{P_n,r_n}  \right| 
		 \]      	
		tends to zero as $n\rightarrow \infty$.  Therefore $\hat{\mu}$ is $U$-invariant.                                    
		   
		\noindent \textbf{Claim: Any U-invariant measure $\hat{\mu}$ on $F$ must be the volume measure} 
		
			By Ratner's measure classification theorem, $\mu$ is supported on a union of closed orbits of subgroups $H$ of $PSL(2,\mathbb{C})$ containing $U$.  We claim that any such $H$ must be equal to $PSL(2,\mathbb{C})$, which  we check by ruling out intermediate candidates for $H$ one by one.

			First of all, $U$ has no closed orbit in $F$. This is because no point on a $U$-orbit in $PSL(2,\mathbb{C})$ is mapped to another point on the same orbit by the action of a matrix corresponding to a hyperbolic isometry of $\mathbb{H}^3$. To see this, let $T$ be a hyperbolic element of $PSL(2,\mathbb{C})$.  If $T$ mapped some frame to another frame on the same $U$-orbit, then for some $u\in U$, $T \circ u$ would fix some frame, and therefore have to be the identity, a contradiction. Therefore, no points on a $U$-orbit can be identified when modding out by $\Gamma$, and $H$ must properly contain $U$. Similar reasoning shows that $H$ cannot be the (projection to $PSL(2,\mathbb{C}$) of) the group of matrices of the form 
			
			\[
			\left\{ \begin{matrix} \begin{pmatrix} 1 & z \\ 0 & 1 \end{pmatrix} \end{matrix} : z \in \mathbb{C}\right\},   
			\]
			whose orbits are horospheres.

			Our assumption on the absence of properly immersed totally geodesic surfaces rules out $PSL(2,\mathbb{R})$ as a possibility for $H$.  It also rules out the group of real upper triangular matrices in $PSL(2,\mathbb{C})$.  An orbit of this group corresponds to a pair consisting of a totally geodesic plane $\Pi$ and a point $p \in \partial_{\infty} \Pi$ on the sphere at infinity of $\mathbb{H}^3$. Such an orbit consists of all frames lying over the plane $\Pi$, such that the first vector of the frame is tangent to a geodesic ray with endpoint $p$ on the sphere at infinity when lifted to the universal cover. Because this set lies over a totally geodesic plane in $\mathbb{H}^3$, it cannot have a closed projection to $F$. One can rule out the subgroup of complex upper triangular matrices on similar grounds.

			Since we've exhausted all conjugacy classes of intermediate closed subgroups (see \cite[Section 4.6]{kapovichbook}), $H$ must be equal to $PSL(2,\mathbb{C})$. This implies that $\hat{\mu}$ equals the Haar measure on $F$, which contradicts (\ref{equiF}.)    
			\end{proof} 
			
			Let  $\mu_{D,P,r}$ be the measure obtained by averaging over $D_{P,r}$.  The next corollary follows by integrating in polar coordinates.     
			
			\begin{cor} \label{diskscor}

			   For any continuous $f$ and fixed $\epsilon>0$ there exists $R_0$ such that for all $P \in Gr_2(M)$ and all $R>R_0$, 
				\begin{equation}\label{equi3}  
					\left| \int_{Gr_2(M)} f d \mu_{D,P,R} - avg(f) \right|< \epsilon,  
				\end{equation} 
				\end{cor}

			We can now prove equidistribution for certain sequences of minimal surfaces in $M$. First we prove a lemma.   
			
				\begin{lem} \label{fubinilemma} 
					Let $R>1$ be given. Then for every $\epsilon$ there is some $\delta$ so that the following is true. Let $\Sigma$ be a closed Riemannian surface % (of any genus greater than 1)
					  with Gauss curvature everywhere in the interval ($-1-\delta,-1 + \delta)$.  Let $f$ be a function on $\Sigma$, and $f_R$ the function that at each point is equal to the average of $f$ over the disk of radius $R$ at that point.  If the injectivity radius of $\Sigma$ is less than $R$ at that point, we define this average by lifting the disk of radius $R$ at that point to the universal cover of $\Sigma$ and taking the average of the pullback of $f$ over this disk. 
					
					Then the averages $\text{avg}(f)$ and $\text{avg}(f_R)$ of these two functions over $\Sigma$  satisfy 
					\begin{equation} \label{almost equal} 
						\left|\text{avg}(f)- \text{avg}(f_R)\right| < \epsilon \max(|f|)	. 
					\end{equation}

				\end{lem}  
				
				\begin{proof}
%We give the proof first in the case that the injectivity radius of $\Sigma$ is everywhere greater than $R$. 

 Let $\Sigma_0$ be the complement of the 1-skeleton of the standard $4g$-gon cell structure on $\Sigma$, over which the tangent bundle to $\Sigma$ is trivial.  Define a metric on the product of a disk $D$ and $\Sigma_0$  as follows.   For fixed $p\in \Sigma_0$, we identify $D$ with the disk of radius $R$ in $T_p\Sigma$, and pull back the metric on $\Sigma$ under the exponential map to get a metric on $D$, whose area form we denote by $dV_{D,p}$.  This metric varies smoothly in $p$, so we can define a smooth metric on $D \times \Sigma_0$ such that the induced metric on each $\{ d \} \times \Sigma_0 \subset D \times \Sigma_0$ is isometric to the metric on $\Sigma_0$, whose area form we denote by $dV_{\Sigma_0}$.   The resulting volume form on $D \times \Sigma_0$ at $(x,y)$ splits as  
\[
 dV_{D,y}(x) \wedge dV_{\Sigma_0}(y).
 \]
 
 For fixed $y$, we identify $D \times \{y\}$ with the disk of radius $R$ in the hyperbolic plane via the exponential map of the centerpoint of $D$, and we write  
\[
dV_{D,y}(x) = \phi(x,y) dV_{\mathbb{H}^2}(x),      
\]			
where $dV_{\mathbb{H}^2}(x)$ is the hyperbolic area form on $D \times \{y\}$ under this identification.  Given $f$ as in the theorem, we define
\[
\hat{f}: D \times \Sigma_0 \rightarrow \mathbb{R} 
\]     
by setting $\hat{f}(x,y)$ to be the value of $f$ at the point on $\Sigma$ that is the image of $x$ under the natural map from $D \times \{y\}$ to $\Sigma$.   We have that,

	\begin{align} \label{fubini} 
		\int_{\Sigma} f_R &=  \int_{\Sigma_0} \frac{1}{ \text{Area}(D(y,R))} \int_{D} \hat{f}(x,y) \phi(x,y)   dV_{\mathbb{H}^2}((x)) dV_{\Sigma_0}((y)), %\\ \nonumber 
	 	%&=   \int_{D} \int_{\Sigma_0}  \frac{1}{ \text{Area}(D(y,R))}\hat{f}(x,y) \phi(x,y)   dV_{\Sigma_0}(y) dV_{\mathbb{H}^2}(x).  
	\end{align}	
\noindent where $\text{Area}(D(y,R))$ is the area of the disk of radius $R$ centered at a lift of $y$ to the universal cover of $\Sigma$.  By taking $\delta$ small enough we can make $\phi(x,y)$ pointwise arbitrarily close to 1, uniformly over all $\Sigma$ satisfying the hypotheses of the theorem.  We can thus also make $\text{Area}(D(y,R))$ arbitrarily close to the area of the disk of radius $R$ in $\mathbb{H}^2$.

%For every fixed $x$, we can therefore make the integrand of the last term in (\ref{fubini}) as close to  $\text{avg}(f) \text{vol}(\Sigma)$, which completes the proof of the lemma. %in the case that the injectivity radius of $\Sigma$ is pointwise greater than $R$. 

Define the function $m_x(y)$ to be the number of distinct geodesic segments of length less than $R$ joining $x$ to $y$.  %The integral of $m_x$ over $\Sigma$ is just the area of the the disk of radius $R$ centered at a lift of $x$ to the universal cover of $\Sigma.$ 
By partitioning $D$ in its hyperbolic metric and $\Sigma_0$ into small almost-Euclidean rectangles and taking Riemann sums for the double integral in (\ref{fubini}), we see that the contribution of $f(x)$ to the integral is weighted by the quantity 
\[
\int_\Sigma m_{x}, 
\]
which can be made as close as desired to the area of the disk of radius $R$ in the hyperbolic plane, since the integral of $m_x$ over $\Sigma$ is just the area of the lift of the disk of radius $R$ at $x$ to the universal cover of $\Sigma$. It follows that the quantity in (\ref{fubini}) can be made $\epsilon \max{|f|}$-close to the integral of $f$ over $\Sigma$ by taking $\delta$ small.

%\[
%\frac{\int_\Sigma m_{x_1}}{\int_{\Sigma} m_{x_2}} 
%\]
%is $O(\delta)$ close to 1 uniformly in $x_1$ and $x_2$.  The proof in the case that the injectivity radius is $\Sigma$ is less than $R$ at some points can then proceed as above up to \ref{fubini}.  Then $f(x)$ will contribute to   

% Since $f(x)$ contributes to the integral of $f_R$ over $\Sigma$ with weight given by $\int m_x$, the lemma can be shown to hold in this case too.  

%arbitrarily small pointwise and we can also make $|\text{Vol}(D(x,r))-$ 

%  we have that the last term in \ref{fubini} differs from    
%\[
%\left|  \int_{D} \int_{\Sigma} \hat{f}(x,y) \phi(x,y)   dV_{\Sigma}((x,y)) dV_{\mathbb{H}^2}((x,y)) - vol_{\mathbb{H}^2}(D) \text{volume}(\Sigma) \text{avg}(f)  \right| < \epsilon \text{vol} \sup  |f| ,    
%\]

%The inequality (\ref{fubini}) will hold if for every point $x$ in  $\Sigma$, $\delta$ is chosen so small that, if the disk of radius $R$ at $x$ in $\Sigma$ and the disk of radius $R$ in the hyperbolic plane are identified via exponential maps, the quotient of the two area forms is pointwise contained in the interval $(1 - \frac{\epsilon'}{2},1 + \frac{\epsilon'}{2})$. 

\end{proof}

\begin{thm} \label{constcurvature} 
	
	Let $M$ be a closed hyperbolic 3-manifold with no properly immersed totally geodesic surfaces, and let $f$ be a continuous function on $Gr_2(M)$ with average $avg(f)$.  Then for every $\epsilon>0$ we can find $\delta$ such that the following holds. 
	%Fix a round circle $C$ on the sphere at infinity.
	Take any surface subgroup of $\pi_1(M)$ realized  by a properly immersed minimal surface $\Sigma$.  Assume also that the limit set of a lift of $\Sigma$ to the universal cover is a $K$-quasicircle for  $K<1 +\delta$. The surface $\Sigma$ includes in $Gr_2(M)$ by its tangent planes, and we define $avg(f,\Sigma)$ to be the average of the pullback of $f$ over $\Sigma$ in the metric on $\Sigma$ induced by $M$. Then        
	
	\[
	\left|	avg(f) - avg(f,\Sigma)   \right| < \epsilon.  
	\]

\end{thm}

%\begin{rem} 
%	Kahn-Markovic showed that any circle in the boundary at infinity is a limit of quasicircle limit sets of surface subgroups, so surfaces of injectivity radius greater than $R$  can be obtained by taking quasicircles close to circles disjoint from the set of all endpoints of lifts of geodesics in $M$ of length less than $2R$.  This ensures that any closed geodesic on the corresponding minimal surface will have length at least $2R$.
%\end{rem}

	\begin{proof}
		For  $\delta$ small enough, \cite{seppi}  implies that $\Sigma$ is unique and has principal curvatures pointwise of magnitude $O(\delta)$.  Let $R$ be larger than the $R_0$ given by Corollary \ref{diskscor} applied to $f$ and $\epsilon/4$.   By making $\delta$ small enough, we can ensure that lifts $D$ of intrinsic disks of radius $R$ in $\Sigma$ to the universal cover are as $C^1$ close as desired to totally geodesic disks of that radius in the universal cover. In particular, for any lift of an intrinsic disk $D$ in $\Sigma$ we can find a totally geodesic disk $D'$ in $\mathbb{H}^3$ such that the averages of the pullback of $f$ over $D$ and $D'$ differ by at most $\epsilon/4$, and the average over $D$ therefore differs from the average of $f$ over $Gr_2(M)$ by at most $\epsilon/2$.

		%Let $D_0$ be the set of points $x$ in $D$ such that there is some other point in $D$ that projects to the same point in $M$. 

	%	For an immersed surface $\Sigma$ in $M$, let $I_{\Sigma}$ be the function that assigns to each point of $\Sigma$ the injectivity radius of $\Sigma$ at $x$ in the metric induced by $M$.  We claim that for any fixed $R$, 
	%	, where $\delta$ is such that $\partial_{\infty} \Sigma$ is a $K=1+\delta$ quasicircle.     

		 Now if $M$ is the maximum of $|f|$ over $Gr_2(M)$, taking the $\epsilon$ of the previous lemma to be  $\epsilon/(2M)$ and making $\delta$ small enough finishes the proof.

	\end{proof}

	\subsection{Variable Curvature} 
	
	We now prove Theorem \ref{quantitativeintro}.  Let $g$ be a metric on $M$ to which Theorems \ref{mainintro} or \ref{main} apply to construct a foliation $\mathcal{F}_g$ conjugate to the totally geodesic foliation in constant curvature via 	
	\[
	\Phi: Gr_2((M,g_{hyp})) \rightarrow Gr_2((M,g)).  
   \]

   Let $\Sigma_n$ be a sequence of properly immersed minimal surfaces in $(M,g_{hyp})$ that satisfy the hypotheses of Theorem \ref{quantitativeintro}: i.e., with areas tending to infinity such that the limit sets of the $\pi_1(\Sigma_n)$ in $\partial_\infty \mathbb{H}^3$ are $K_n$-quasicircles with $K_n$ tending to 1. Assume that the $\Sigma_n$ become uniformly distributed in $Gr_2(M)$.  In the previous subsection, we checked that this is always the case if $M$ contains no closed totally geodesic surfaces.  We will show that sequences of $g$-minimal surfaces $\Sigma_n'$ corresponding to the $\Sigma_n$ converges to a measure $\mu_g$ with  full support. 
   
   First we describe the measure $\mu_g$.  To do so, we need to introduce some terminology.  A foliation $\mathcal{G}$ with total space $E$ is given by a collection of charts of the form 
   \[
   \phi_\alpha: U_\alpha \times T_\alpha \rightarrow E,
   \]
   \noindent where $\phi_{\alpha}$ is a homeomorphism onto its image,  $\phi_{\alpha}|_{U_\alpha \times \{t_{\alpha}\}}$ maps into a leaf of $\mathcal{G}$ for any fixed $t_\alpha \in T_\alpha$ , and $U_\alpha$ and $T_\alpha$ are both homeomorphic to balls.  We call $T_\alpha$ the \textit{transversal} for the local chart defined by $\phi_{\alpha}$. Any closed loop in the total space beginning and ending at a point in the chart defines a \textit{transverse holonomy map} $T_\alpha \rightarrow T_{\alpha}$.

    \begin{defn} 
   	A \textit{transverse invariant measure} for a foliation is a choice of measure on the transversal for each local chart for the foliation, well-defined with respect to transition maps between overlapping charts, that is invariant under transverse holonomy maps.
   \end{defn}

The model foliation $\mathcal{F}=\mathcal{F}_{g_{hyp}}$ of $Gr_2(M,g_{hyp})$ has a natural tranverse invariant measure $\mu_T$ because its leaves are the orbits of an $SL(2,\mathbb{R})$ action, and $SL(2,\mathbb{R})$ is unimodular.  \iffalse This is because the determinant of the derivative of the action restricted to the transversal defines a homomorphism $SL(2,\mathbb{R}) \rightarrow \mathbb{R}^+$, which must be trivial since $SL(2,\mathbb{R})$ is unimodular. \fi The natural volume form $\mu_{g_{hyp}}$ for $Gr_2(M,g_{hyp})$, normalized to have unit volume, decomposes in a local chart for $\mathcal{F}$ as a transverse sum of $\mu_T$ and the area forms for the leaves, up to scaling by a constant.  Taking the pushforward of $\mu_T$ by the map $\Phi$ that conjugates $\mathcal{F}$ to $ \mathcal{F}_g$ and summing this transverse invariant measure with the area forms for the leaves of $\mathcal{F}_g$ in the induced metrics for their projections to $(M,g)$ gives a measure on $Gr_2(M,g)$.  We define $\mu_g$ to be this measure, normalized to have unit volume.  

It is clear from the construction of $\Phi$ in terms of normal geodesic projections that $\Phi$ restricted to any leaf of $\mathcal{F}$ is smooth, and the leafwise differential of $\Phi$ varies continuously in the total space of $\mathcal{F}$.  We can thus define a continuous function $f_{\Phi} : Gr_2(M,g_{hyp}) \rightarrow \mathbb{R}$ by setting $f_{\Phi}(x)$ equal to the determinant of the differential at $x$ of $\Phi$ restricted to the leaf of $\mathcal{F}$ through $x$.  We can write this as
\begin{equation} \label{correctedpshfwd} 
 \Phi_* \mu_{g_{hyp}} = f_{\Phi} \circ \Phi^{-1}   \cdot \mu_g.  
 \end{equation} 
 
Now let $\Sigma_n$ be the sequence of closed  immersed surfaces in $(M,g_{hyp})$ from the start of the section.  Let $\Sigma_n'$ be the corresponding minimal surfaces in $(M,g)$. Recall the maps $f_{\Sigma_n}: \Sigma_n \rightarrow \Sigma_n'$ from the proof of \ref{uniqintro}, that exist for $n$ sufficiently large.  As $n$ tends to infinity, $\Sigma_n$ and $\Sigma_n'$ locally converge to (projections of) leaves of $\mathcal{F}$ and $\mathcal{F}_g$, and the maps $f_{\Sigma}$ locally converge to the restriction of $\Phi$ to the $\Sigma_n$.

Let $\mu_n'$ and  $\mu_n$ be the probability measures on $Gr_2(M,g)$ and $Gr_2(M,g_{hyp})$ corresponding to the $\Sigma_n'$ and the $\Sigma_n$.   We claim that the measures $\mu_n'$ weak-$*$  converge to $\mu_g$.  This is almost immediate from how everything has been set up.  Since the $f_{\Sigma_n}$ are locally converging to $\Phi$, the pushforwards of the $\mu_n$ by the $f_{\Sigma_n}$ are locally converging to the $\mu_n'$ multiplied by $f_\Phi \circ \Phi^{-1}$. By assumption the $\mu_n$ are converging to the volume measure $\mu_{g_{hyp}}$ on $Gr_2(M,g_{hyp})$.  Therefore by (\ref{correctedpshfwd}), the $\mu_n'$ are converging to $\mu_g$.  This finishes the proof of Theorem \ref{quantitativeintro}.

 We note that the only place where we used the fact that the $\Sigma_n'$ and the projections to $(M,g)$ of the leaves of $\mathcal{F}_g$ were minimal surfaces was in showing that the $\Sigma_n'$ locally converged to leaves of $\mathcal{F}_g$ and that the $f_{\Sigma_n}$ locally converged to $\Phi$.

\section{Examples where Theorem \ref{mainintro} Cannot Apply} \label{counterexample} 

\subsection{Introduction} 
We now present examples of closed hyperbolic 3-manifolds $M$ and negatively curved metrics $g$ on $M$ for which $Gr_2(M)$ cannot possibly admit a foliation as in  Theorem \ref{mainintro} or Theorem \ref{main}.  Recall that $\mathcal{P}_{\tilde{g}}$ is the set of totally geodesic planes in $\mathbb{H}^3$ considered as subspaces of $(\tilde{M},\tilde{g})$.  In the examples we construct, there will be multiple properly embedded minimal planes in $(\tilde{M},\tilde{g})$ at finite Hausdorff distance from the same element of $\mathcal{P}_{\tilde{g}}$, and so Theorem \ref{mainintro} could not apply to produce a foliation. The starting point for our construction will be a closed hyperbolic 3-manifold $M$ with a proper embedded totally geodesic surface $\Sigma$ in its hyperbolic metric (see \cite{ah3mtext} for examples.)  
\iffalse 
We will need the following lemma, which is contained in \cite[Theorem 7.5]{warpedproduct}. 	 
\begin{lem} \label{warped} 
 Let $f:(a,b) \rightarrow \mathbb{R}^+$ be a smooth strictly convex function, and let $h$ be a negatively curved metric on a surface $\Sigma$ of genus $g>1$.  	Then the  warped product metric 
	\[
	 f^2(t)h + dt^2  
	\]
on $\Sigma \times (a,b)$ has negative sectional curvature.  
	\end{lem} 
\fi 
\subsection{Modification of the Quasi-Fuchsian Metric Near Infinity} 
  In \cite{zenononuniq}, examples of quasi-Fuchsian hyperbolic 3-manifolds were constructed with arbitrarily many stable embedded minimal surfaces in their convex core whose inclusions are homotopy equivalences.  Fix such a quasi-Fuchsian manifold $Q$.   
  
   %Let $g_1$ and $g_2$ be hyperbolic metrics on $\Sigma$ in the conformal class of the ends of $Q$. 
  Let $2I_1^*$ and $2I_2^*$ be hyperbolic metrics in the conformal classes of the two ends of $Q$.   In  \cite[Theorem 5.8]{krasnovschlenker} it is shown that there are neighborhoods of infinity of each of the ends of $Q$, homeomorphic to $\Sigma \times (0,\infty)$, on which there are coordinates such that the metric on $Q$ can be written, for $t$ larger than some $T>>1$, in the form 
  \begin{equation} \label{nbdinfty} 
   \frac{1}{2}(e^{2t} I_j^* + 2 II_j^* + e^{-2t} III_j^*) + dt^2 \hspace{2mm}  j=1,2, 
  \end{equation}     

\noindent where the symmetric 2-tensors  $II^*_j$ and $III^*_j$ are determined by the $I^*_j$. There thus exist neighborhoods of each of the ends that have foliations by negatively curved equidistant surfaces, which are given by taking $T$ large enough and setting $t$ to be constant in the above coordinates.  These foliations were first constructed by Epstein \cite{epsenv}.

%These neighborhoods are foliated by equidistant surfaces $t=const$.  
%These coordinate charts were constructed by Epstein in cite{epsenv}.    

 %This is achieved by considering the parallel distance-$t$ surfaces to a strictly convex surface $\Sigma_j$ outside of the convex core and with principal curvatures less than 1.  The metric $I^*_j$ %is in the conformal class of $g_j$ and 
%can be expressed in terms of the first, second, and third fundamental forms of $\Sigma_j$.  It has strictly negative curvature by \cite[Lemma 5.2]{krasnovschlenker}.    

  \begin{lem} \label{warpnegcurv}  Let $g_t$, $t\in [0,1]$, be a smooth family of metrics on the surface $\Sigma$.  Then there exists $T$ such that the warped product metric 
  	\[
  	\frac{1}{4}e^{2t} g_{t} + dt^2 
  	\]
  	\noindent for $t \in (T,T+1)$ has negative sectional curvature. The number $T$ depend only on norms of the family $g_t$ and its time derivatives relative to some fixed background metric. 
  	
  	\end{lem} 
  
  This is not hard to check from the formula for the sectional curvature in terms of the Christoffel symbols, and the formula for the Christoffel symbols in terms of derivatives of the metric.  The terms that contain $t$-derivatives of the metric are much larger than the other terms and control the sign.  
  
  Now, for $t_0$ large, take a smooth path of metrics $\gamma_t^j$, $t\in[0,1]$, joining
\begin{equation} 
2(I^*_j + 2e^{-2t_0} II^*_j + e^{-4t_0} III^*_j)
\end{equation} 
\noindent to 
\[ 
(1+ 2e^{-2(t_0+1)} + e^{-4(t_0+1)})g_{\Sigma}= 4e^{-2(t_0+1)}\cosh^2(t_0+1) g_{\Sigma},  
\]
\noindent where $g_{\Sigma}$ is the induced hyperbolic metric on the totally geodesic surface $\Sigma$ in $M$. We also require that for all $\epsilon$ small, 
\[
\gamma_{\epsilon}^j =2( I^*_j + (2e^{-2(t_0+\epsilon)} II^*_j + e^{-4(t_0+\epsilon)} III^*_j))
\]
\noindent and
\[
\gamma_{1-\epsilon}^j=  (1+ 2e^{-2((t_0+1-\epsilon)} + e^{-4(t_0+1-\epsilon)})g_{\Sigma}. 
\]
\noindent Clearly the paths $\gamma_t^j$ can be chosen so that the norms of the metrics $\gamma_t^j$ and their derivatives in time are uniformly bounded, independent of $t_0$, relative to a fixed background metric.  Then applying Lemma \ref{warpnegcurv}, provided we chose $t_0$ large enough, the metric
\begin{equation} \label{interpol} 
\frac{1}{4}  e^{2t} (\gamma_{(t-t_0)}^j) + dt^2 
\end{equation} 
\noindent on $\Sigma \times (t_0,t_0+1)$ has negative sectional curvature.  

\subsection{Modification of the Metric on a Fuchsian Cover} 
Now take $M$ to be any closed hyperbolic 3-manifold that admits a closed embedded totally geodesic surface $\Sigma$. Let $F\cong \Sigma \times \mathbb{R}$ be the cover of $M$ corresponding to $\Sigma$.  If $g_{\Sigma}$ is the metric on $\Sigma$ induced by the hyperbolic metric on $M$, the metric on $F$ can be written as a warped product
\begin{equation} \label{fuchsian}  
 (\cosh^2t) g_{\Sigma} + dt^2.   
\end{equation}  

We can define a new negatively curved metric on $F$ by cutting out the $t_0+1$-neighborhood of $\Sigma$ and gluing in the metric (\ref{interpol})  from the previous subsection, which interpolates between the metrics on the ends of $F$ and the middle of $Q$. This new metric on $F$ contains a region isometric to the convex core of $Q$.    Denote this new metric on $F$ by $g'$.

\iffalse 
The space of hyperbolic metrics on a surface is path-connected, so we can choose smooth paths $\gamma_1$ and $\gamma_2$ in the space of hyperbolic metrics on $\Sigma$ joining $g_{hyp}$ to $2I^*_1 $ and $2I^*_2$.

For $1<<T_1<T_2$  we can use $\gamma_1$ and $\gamma_2$ to modify (\ref{fuchsian}) on $\Sigma \times (T_1,T_2)$ and $\Sigma \times (-T_2,-T_1)$ so that near $(-1)^{j+1} T_1$ it is equal to 
\begin{equation} \label{fuchsianmodified} 
 \frac{1}{4}e^{2t}(2I^*_j) + dt^2   \hspace{3mm} j=1,2,    
\end{equation}  
and near $\pm T_2$ it remains equal to (\ref{fuchsian}.)  By taking $T_2-T_1$ large,  this can be done so that the new metric has negative sectional curvature by Lemma \ref{warped} and the reasoning of the previous subsection. %is locally $C^2$-close to a warped product of a negatively curved metric with convex warping function as in (\ref{warped}), so that the sectional curvatures of the new metric remain negative.   

Provided $T_1$ was taken large enough, (\ref{fuchsianmodified}) agrees with (\ref{modified}) in neighborhoods of $\Sigma \times \{ \pm T_1 \}$, so we can define a new negatively curved metric on $F$ by cutting out a region $\Sigma \times [-T_1 + \epsilon, T_1 - \epsilon]$ and gluing in a region of $Q$ containing its convex core. \fi

\subsection{Passage to Finite Covers to Facilitate Gluing } \label{covers} 
We claim that by passing to finite covers, we can make the normal injectivity radius of $\Sigma$ in $M$ arbitrarily large. It is a fact that for every $g \in \pi_1(M) \setminus \pi_1(\Sigma)$, there exists a finite index subgroup $G$ of $\pi_1(M)$ such that $\pi_1(\Sigma)\subset G$ but $g \notin G$ \cite[Lemma 5.3.6]{ah3mtext}.

Fix a basepoint for $\pi_1(M)$, and let $F$ be the Fuchsian cover of $M$ corresponding to $\pi_1(\Sigma)$. Fix a connected polyhedral fundamental domain $P$ for the action of $\pi_1(M)$ on  $\mathbb{H}^3$.  Then $F$ is tessellated by copies of $P$ which are fundamental domains for the covering map from $F$ to $M$, and any fixed normal neighborhood $N$ of the central totally geodesic copy of $\Sigma$ in $F$ is contained in a finite number copies of $P$ in $F$. By choosing elements of $\pi_1(M)$ representing the cosets of $\pi_1(\Sigma)$ corresponding to these finitely many fundamental domains, we can find a finite index subgroup of $\pi_1(M)$ containing $\pi_1(\Sigma)$ but not containing any of these coset representatives.  The finite cover $M'$ corresponding to this subgroup is then also covered by $F$, and the projection of the normal neighborhood $N$ to $M'$ is injective.  This shows that the normal injectivity radius of $\Sigma$ can be made arbitrarily large by passing to finite covers of $M$.

\subsection{Construction of the Metric} 

Outside of a large central neighborhood $N$ in $F$, the metric   $g'$ agrees with the Fuchsian metric. By \ref{covers}, we can find a finite cover $M'$ of $M$ such that every point on $N$ is at a distance from $\Sigma$ in the Fuchsian metric less than the normal injectivity radius of $\Sigma$ in $M'$. Then $N$ projects injectively to $M'$ under the covering map, so we can use $g'$ to define a new negatively curved metric on $M'$ which we also call $g'$.  Outside of the projection of $N$ to $M'$, $g'$ is equal to the original hyperbolic metric on $M'$.  

In this new metric $g'$ on $M'$, there are several properly embedded $\pi_1$-injective stable  minimal surfaces whose fundamental groups include as subgroups in the conjugacy class of $\pi_1(\Sigma)$ in $\pi_1(M')$.  These lift to the universal cover of $M'$ to give several distinct properly embedded minimal planes at finite Hausdorff distance from the same element of $\mathcal{P}_{\tilde{g}'}$. %the lift of $\Sigma$ to the universal cover. 

\begin{rem}

It seems likely that the metric on $M'$ we constructed can be joined to the constant curvature metric through a smooth path of negatively curved metrics by performing the above construction on a smooth path in quasi-Fuchsian space joining $F$ to $Q$.  Theorem \ref{mainintro} would then apply with $T<\infty$ to this path of metrics.     	
	
\end{rem}

\section{A Stability Estimate for the Foliations of Theorem \ref{mainintro}} \label{stabilitysect}

Fix a smooth family of metrics $g_t$, $t\in[0,1]$, to which Theorem \ref{mainintro} applies with $T=\infty$ to produce foliations. Let $S$ be a totally geodesic plane in $\tilde{M} \cong \mathbb{H}^3$, and let $S_t$ be the $\tilde{g}_t$-minimal plane in $\tilde{M}$ corresponding to $S$.  For fixed $t_0\in[0,1)$, elliptic PDE theory implies that $S_{t}$ converges smoothly uniformly to $S_{t_0}$. For $t$ close to $t_0$, $S_t$ is a graph over $S_{t_0}$ in normal coordinates for a tubular neighborhood of $S_{t_0}$, and so differentiating in $t$ at $t_0$ we obtain a vector field $v$ normal to $S_{t_0}$.  %I thank my advisor Fernando Coda Marques for suggesting that I prove a theorem along the lines of the one below and providing the basic idea to prove it.     

\begin{thm} \label{stability}  
 Let $\epsilon_0 < \epsilon$ be given, and suppose $S_{t_0}$ is $\epsilon$-subordinate.  Then there exists $\delta$ depending only on $\epsilon$ and bounds on the $g_t$ and the time derivatives $g_t'$ of the $g_t$ such that if $\delta> t-t_0>0$, then $S_t$ is $\epsilon_0$-subordinate.  (The same $\delta$ works for every $\epsilon$-subordinate $S_{t_0}$.)        		
\end{thm}   

\noindent By bounds on $g_t$ and $g_t'$, we mean bounds on these tensors and their covariant derivatives up to second order in the $L^\infty$ norms induced by the hyperbolic metric $g_0$.      

\begin{proof}
The minimal surface equation for a graph $u$ in $\mathbb{R}^3$ over a region in the $xy$-plane can be written 
\begin{equation} \label{minsurfaceflat} 
(1+u_{x_2}^2) u_{x_1x_1} + (1 + u_{x_1}^2) u_{x_2x_2} - 2 u_{x_1} u_{x_2} u_{x_1x_2}=0.  	
\end{equation} 

If we dilate the metric and zoom in at a point $x_0$ on $S_{t_0}$ at a scale where the metrics on both $\tilde{M}$  and $S_{t_0}$ are almost Euclidean to second order, then the coefficients $A^{ij}$,$B^i$,$C$ of the second order equation
\[
A^{ij} u_{x_ix_j} +   B^i u_{x_i} + C=0
\]
 $S_{t_0}$ satisfies, writing it locally as a graph $u$ over a coordinate plane, can be made arbitrarily $C^0$ close to those of Equation (\ref{minsurfaceflat}).  The amount we need to dilate the metric at $x_0$ to obtain a given degree of closeness is determined by bounds on  $g_{t_0}$ and the principal curvatures of $S_{t_0}$, and the latter can be bounded in terms of bounds on $g_{t_0}$ by the fact that $S_{t_0}$ is $\epsilon$-subordinate.  For $t$ close to $t_0$ and $x$ close to $x_0$, we can write the surfaces $S_{t}$ as graphs $u(x,t)$, where $u(x,t_0) = u(x)$.  Then the derivative vector $v$ in these coordinates equals  $u_t(x,t_0)$. Differentiating the minimal surface equations the $u(x,t)$ satisfy in time at $t_0$, we obtain a second order equation of the form 
\begin{equation} \label{differentiated} 
a^{ij} v_{x_ix_j} +   b^i v_{x_i} + c=0.   
\end{equation}  

The $a^{ij}$ can be made arbitrarily $C^0$ close to the coefficients of $u_{x_i x_j}$ in Equation (\ref{minsurfaceflat})  provided we dilated the metric enough at $x_0$, and so we can take (\ref{differentiated}) to be uniformly elliptic with ellipticity constant 1/2.  The amount we need to dilate the metric to ensure this again just depends on bounds on  $g_{t_0}$ and the principal curvatures of $S_{t_0}$.   The $b^i$ and $c$ can be bounded in terms of the first two derivatives of $u$--- which in turn are bounded by the absolute values of the principal curvatures of $S_{t_0}$ and bounds on $g_{t_0}$, bounds on $g_{t_0}$, and bounds on $g_{t_0}'$.  It follows that we can estimate the $C^2$-norm of $v$ if we have a $C^0$ estimate for $v$ \cite[Theorem 6.2, Schauder Interior Estimates]{gilbargtrudinger}. 

 By Lemma \ref{lemconvex}, the mean curvatures of the parallel distance-$s$ surfaces to $S_{t_0}$ are greater than $\frac{\epsilon}{2}s$ for $s$ small and positive and less than $\frac{\epsilon}{2}s$ for $s$ small and negative. Taking a coordinate chart where the parallel distance-$s$ surface is a graph $y^s$ over a coordinate plane, we can write the mean curvature of the parallel distance-$s$ surface in the metric $g_t$ as
 \[
 a^{ij}(t,s) y^s_{x_ix_j} +   b^i(t,s) y^s_{x_i} + c(t,s),   
 \]
 for $t$ close to $t_0$. Since 
 \[
 \max( |a^{ij}(t,s)-a^{ij}(t_0,s)|,|b^{i}(t,s)-b^i(t_0,s)|,|c(t,s)-c(t_0,s))< M|t-t_0|
 	\]
 	\noindent for some $M$ just depending on the norm of $g_t$ and its derivative in time and $\epsilon$, we can find some small $\xi$ \iffalse that depends on bounds on the $g_t$\fi  such that if 
\begin{equation} \label{xi}  
 |t- t_0| < \xi|s|,       
\end{equation}                   
then the distance-$s$ parallel surfaces are mean-convex in the metric $g_t$.  The constant $\xi$ can be chosen so that this statement also holds for all of the other $\epsilon$-subordinate (projections of) leaves of the foliation, not just the given $S_{t_0}$ we are considering. One can then show by essentially the same argument as the proof of Lemma \ref{lim} that $S_t$ is contained between the signed distance $\pm s$ parallel surfaces to $S_{t_0}$ if $t$ and $s$ satisfy (\ref{xi}.)  By sending $t$ to $t_0$, this implies that the magnitude of $v$ is bounded above by $1/\xi$ in the coordinates we chose. This gives the desired $C^0$ estimate for $v$.  

In this way, we can obtain an upper bound for the $C^2$ norm of the normal derivative vector field $v(t)$ for $S_t$ which are $\epsilon$-subordinate. We thus obtain an upper bound for the $C^2$ norm of $v(t)$ for $S_t$ that are $\epsilon_0$-subordinate, since such $S_t$ are also $\epsilon$-subordinate.  The surface $S_{t_0}$ is $\epsilon_0$-subordinate, and $S_t$ remains $\epsilon_0$-subordinate for $t-t_0$ less than some $\delta$.  One can obtain a lower bound for $\delta$ in terms of the difference $\epsilon - \epsilon_0$, bounds on $g_t$ and $g_t'$, and the bound on the $C^2$ norm of $v(t)$ by differentiating the formula for the principal curvatures of $S_t$ with respect to time in a local coordinate chart, and using the $C^2$ bound on $v(t)$.  This completes the proof.

\end{proof}

		\bibliography{bibliography}{}
		\bibliographystyle{alpha}
	
\end{document}